\documentclass{amsart}
\usepackage{amssymb,epsfig}
\usepackage{color}
\usepackage{bbold}
\usepackage{stmaryrd}

\newtheorem{theorem}{Theorem}[section]
\newtheorem{lemma}[theorem]{Lemma}
\newtheorem{proposition}[theorem]{Proposition}

\theoremstyle{definition}

\theoremstyle{remark}
\newtheorem{remark}[theorem]{Remark}

\numberwithin{equation}{section}

\newcommand{\eps}{\varepsilon}

\newcommand{\R}{\mathbb R}
\newcommand{\N}{\mathbb N}

\newcommand{\cB}{\mathcal B}
\newcommand{\cP}{\mathcal P}
\newcommand{\bbf}{\mathbf{b}}
\newcommand{\bn}{\mathbf{n}}

\newcommand{\Lis}{\cL\mathrm{is}}

\newcommand{\bb}{{\bf \underline{b}}}
\newcommand{\cc}{\underline{c}}
\newcommand{\dd}{\underline{d}}

\DeclareMathOperator*{\argmin}{argmin}

\DeclareMathOperator{\clos}{clos}

\DeclareMathOperator{\Span}{span}
\DeclareMathOperator{\diam}{diam}

\DeclareMathOperator{\divv}{div}

\newcommand{\U}{\mathbb{U}}
\newcommand{\V}{\mathbb{V}}

\newcommand{\cL}{\mathcal{L}}
\newcommand{\cR}{\mathcal{R}}
\newcommand{\cT}{\mathcal{T}}

\newcommand{\bx}{\mathbf{x}}
\newcommand{\bs}{\mathbf{s}}

\newcommand{\by}{\mathbf{y}}

\newcommand{\be}{\begin{equation}}
\newcommand{\ee}{\end{equation}}

\newcommand{\TT}{\mathbb{T}}

\newenvironment{Array}{
 \everymath{\displaystyle\everymath{}}
 
 \array
 }
{\endarray }

\title{On the stability of DPG formulations of transport equations}\thanks{The first author has been supported by the Netherlands Organization for Scientific Research (NWO) under contract. no. 613.001.109, the second author has been supported in
part by the  DFG SFB-Transregio 40, by the DFG Research Group 1779, and
the Excellence Initiative of the German Federal and State Governments}

\date{\today}

\author{D. Broersen, W. Dahmen, R.P. Stevenson}
\address{
Korteweg-de Vries Institute for Mathematics,
University of Amsterdam,
P.O. Box 94248,
1090 GE Amsterdam, The Netherlands}
\address{Institut f\"ur Geometrie und Praktische Mathematik, RWTH Aachen, Germany}
\email{dirkbroersen@gmail.com, dahmen@igpm.rwth-aachen.de, r.p.stevenson@uva.nl}



\begin{document}

\begin{abstract} 
In this paper we formulate and analyze a Discontinuous Petrov Galerkin formulation of linear transport equations with
variable convection fields. We show that a corresponding {\em infinite dimensional} mesh-dependent variational formulation, in which besides the principal field also its trace on the mesh skeleton is an unknown, is uniformly stable with respect to the mesh, where the test space is a certain product space over the underlying domain partition.

Our main result states then the following.  For piecewise polynomial trial spaces of degree $m$, we show under mild assumptions on the convection field that piecewise polynomial test spaces of degree $m+1$ over a  refinement of the primal partition with uniformly bounded refinement depth give rise to uniformly (with respect to the mesh size)
stable Petrov-Galerkin discretizations. The  partitions are required to be shape regular but need not be quasi-uniform.
An important startup ingredient is  that for a constant convection field  one can identify  the exact optimal test functions with respect to 
a suitably modified but uniformly equivalent broken test space norm as piecewise polynomials. These test functions are then varied
towards simpler and stably computable near-optimal test functions for which the above result is derived via a perturbation analysis.
We conclude indicating some consequences of the results that will be treated in forthcoming work.

\end{abstract}

\maketitle

\section{Introduction}  \label{S1}
There has been a recent vibrant development of the so called {\em Discontinuous Petrov Galerkin} (DPG) method, initiated and developed
mainly by L. Demkowicz and J. Gopalakrishnan, see e.g. \cite{64.14,75.61}. The general underlying methodology aims,
in particular, at an improved treatment of problem classes that are, roughly speaking, much less understood than classical second order elliptic problems.
Of course, ``improved'' leaves much room for interpretation but for us, predominant aspects are the following: 
\begin{itemize}
\item[(i)]
Ideally, even though the original problem may be unsymmetric or indefinite, the arising system matrices are symmetric positive definite
and sparse, so that one has a chance to keep the computational complexity proportional to the problem size.
\item[(ii)]
Ideally, the method is based on a DG-type variational formulation that establishes a tight relation between errors and residuals.
\end{itemize}
We emphasize that we mean in (ii) the {\em outer} residual, i.e., the residual in a full infinite dimensional space where it is well defined.
Once  a suitable topology for this space is identified such a residual can be used as a rigorous foundation for deriving error indicators
that could steer adaptive techniques. Being able to do this beyond the class of elliptic problems is a major motivation for this paper.
Specifically, the central objective of this paper is to discuss (i) and (ii) for a class of {\em linear transport equations} with possibly {\em variable}
convection field.

We explain next the relevance of (i), (ii) for us in more detail, relate our findings to the state of the art, and lay out the objectives of the present work.
\subsection{Conceptual background and motivation}\label{S1.1}
 Both issues (i), (ii) above rely crucially on the notion of {\em optimal test bases}. The key underlying idea is easily described in an abstract
framework and has been presented in the literature in different variants for different purposes \cite{BaMo84,35.856,64.14,75.61,64.573,58.3,DPW}.
To explain this let $\U,\V$ denote
Hilbert spaces over $\R$, endowed with norms $\|\cdot\|_\U, \|\cdot\|_\V$, respectively,
 and assume that $b(\cdot,\cdot): \U \times \V \to \R$ is
a continuous  bilinear form.  Given $f\in \V'$, the normed dual of $\V$, endowed with the norm
$$
\|w\|_{\V'}:= \sup_{v\in \V}\frac{|w(v)|}{\|v\|_\V},
$$

 consider the  variational problem
\be
\label{varprob}
b(u,v) = f(v),\quad v\in \V.
\ee
Since the form $b(\cdot,\cdot)$ is continuous, i.e., 
$$
\|\cB\|:=
 \sup_{\|v\|_\V\le 1}\sup_{\|w\|_\U\le 1} b(w,v) <\infty,
$$
the operator $\cB:\U \to \V'$, defined by $(\cB w)(v)= b(w,v)$, $w\in\U, v\in\V$, is   continuous and 
 \eqref{varprob} is equivalent to the operator equation
\be
\label{opeq}
\cB u = f.
\ee
 Its unique solvability  is well known to be equivalent to the validity of the inf-sup conditions
\be
\label{inf-sup}
\inf_{w\in\U}\sup_{v\in\V}\frac{b(w,v)}{\|w\|_\U\|v\|_\V}\ge \beta ,\quad \inf_{v\in\U}\sup_{w\in\V}\frac{b(w,v)}{\|w\|_\U\|v\|_\V}\ge \beta,
\ee
for some positive $\beta$, i.e.,  $\cB\in \Lis(\U,\V')$ where $\Lis(\mathbb{X},\mathbb{Y})$   denotes  the collection of norm-isomorphisms from a Hilbert space $\mathbb{X}$ onto a Hilbert space
$\mathbb{Y}$.

Moreover, denoting by $\cL(\mathbb{X},\mathbb{Y})$ the space of bounded linear operators from the normed linear space $\mathbb{X}$
 to the  normed linear space $\mathbb{Y}$, it is well known that
 \mbox{$\|\cB^{-1}\|_{\cL(\V',\U)}\le \beta^{-1}$}. Thus,

the {\em condition number} of $\cB\in \Lis(\U,\V')$ 
$$
\kappa_{\U,\V'}(\cB):= \|\cB\|_{\cL(\U,\V')}\|\cB^{-1}\|_{\cL(\V',\U)}
$$
satisfies
$$
\kappa_{\U,\V'}(\cB) \le \|\cB\|/\beta,
$$
i.e., the smaller $\|\cB\|$ and the larger $\beta$, the better.
  In particular,
since in these terms $\|w\|_\U\le \beta^{-1}\|\cB w\|_{\V'}$, $\|\cB w\|_{\V'}\le \|\cB\| \|w\|_\U$, we do have for any approximation
$\bar u$ to the solution $u$ of \eqref{varprob} the error-residual relation
\be
\label{error-res}
\|\cB\|^{-1}\|f-\cB \bar u\|_{\V'} \le \|u-\bar u\|_\U \le \beta^{-1}\|f- \cB \bar u\|_{\V'}.
\ee
Of course, the larger  $\kappa_{\U,\V'}(\cB)$ the harder time has a numerical method based on the above variational formulation   to perform well.
Moreover, the residual in $\V'$ does then not provide accurate information about the error in $\U$.

In general one may have to face two types of obstructions:  first, $\kappa_{\U,\V'}(\cB)$ - although finite - could be very large. A typical example
is a convection dominated convection diffusion problem for $\U=\V = H^1_0(\Omega)$.
Fixing  $\|\cdot\|_\U$ and appropriately varying $\|\cdot\|_\V$, ore vice versa, may lead to a different variational formulation with
a much smaller condition number, ideally even equal to one, see \cite{58.3}.
The prize to be paid is that one has to accept that trial and test space (already on the infinite dimensional level) are different.
This is the second obstruction, namely having to deal with an {\em asymmetric} variational formulation - $\U\neq \V$ - so that
the uniform discrete stability of projected versions of \eqref{varprob} is no longer for granted even though the inf-sup constant $\beta$
in \eqref{inf-sup} may be close to one.

The present paper is concerned with this second issue, starting with a well-conditioned infinite dimensional  variational formulation --
later for a class of   transport
equations. Then, given a (finite dimensional) {\em trial space} $\U^h\subset \U$ we wish to find a {\em test space} $\TT^h  \subset \V$
that inherits the stability \eqref{inf-sup} of the infinite dimensional problem  (for a positive constant possibly  smaller than $\beta$, but $h$-independent), and therefore deserves to be called  (uniformly) ({\em near}--){\em optimal}.
To identify such a  near--optimal test space, notice first that the {\em trial-to-test-map} 
 $\cT \in \Lis(\U,\V)$, defined by 
 
\begin{equation} 
\label{vartestexact}
\langle \cT u , v\rangle_\V=b(u;v) \quad(u \in \U,\,v \in \V),
\end{equation}
yields the {\em supremizer} in the first relation of \eqref{inf-sup}, i.e.,
\be
\label{supremizer1}
\|\cT u \|_\V = \sup_{v\in\V}\frac{b(u,v)}{\|v\|_\V},
\ee
which means
\be
\label{supremizer}
\|\cT u\|_\V^2= b(u,\cT u).
\ee
Therefore, the  (truly) optimal test space for a given subspace $\U^h\subset \U$ is
\be
\label{opttestspace}
  \cT(\U^h)= \{\cT u^h\colon u^h\in \U^h\},
\ee
in the sense that the Petrov-Galerkin scheme: find $u_h\in\U_h$ such that
\be
\label{PG}
b(u_h,v_h)=f(v_h),\quad v_h \in  \cT(\U^h)   ,
\ee
is uniquely solvable and the corresponding finite dimensional operator has  at most the same condition number as the
infinite dimensional problem \eqref{varprob}. Moreover, \eqref{PG} is easily seen to form the normal equations for minimizing the residual 
$\|f-\cB w\|_{\V'}$
over $\U^h$, i.e.,
\be
\label{minres}
u^h=\argmin_{\bar{u}^h \in \U^h} \|f-\cB \bar{u}^h\|_{\V'}.
\ee

Denoting by
$\cR_\U\in \Lis(\U, \U')$  the {\em Riesz-map} defined by
\be
\label{Riesz-lift}
\langle z, w\rangle_\U = (\cR_\U z)(w),\quad z , w\in \U,
\ee
we have, of course, $\cT = \cR_\V^{-1}\cB = \cR_{\V'}\cB$. Hence, the application of $\cT$ amounts to solving an
infinite dimensional Galerkin problem in $\V$. Thus,
for each basis function $\phi\in \U^h$,  finding the corresponding test-basis function $\psi= \cT\phi$,
would require solving an infinite dimensional variational problem, possibly even of
the same complexity  as the one for solving \eqref{varprob}. 

A natural idea propagated in many works  (see e.g. \cite{64.14,45.44,58.3,35.856})
is to reduce this $\V$-projection to a finite dimensional subspace $\V^h\subset \V$ which we refer to as the {\em test-search-space}. Specifically,
this amounts to replacing $\cT$ by the mapping $\cT^h=  \cT^{\V^h}  \in \cL(\U,\V^h)$,

defined by 
\begin{equation} \label{vartest}
\langle \cT^h u , v^h\rangle_\V=b(u;v^h) \quad(u \in \U,\,v \in \V^h),
\end{equation}
whose existence is guaranteed by Riesz' representation theorem. 
Given a closed linear trial space $\U^h \subset \U$, and denoting by $\mathcal{P}_{\V^h}$ the $\V$-orthogonal projection  onto $\V^h$, defined by $\langle \mathcal{P}_{\V^h}v,z\rangle_\V = \langle v,z\rangle_\V$, $v \in  {\V}\,,z\in \V^h$, we see that $\cT^h=\mathcal{P}_{\V^h}\circ\cT$.
The range of its restriction to $\U^h$

$$
  \cT^h(\U^h) = (\mathcal{P}_{\V^h}\circ\cT)(\U^h), 
  $$
 known as the {\em projected optimal test space},  will now be used as test space in  the {\em Petrov-Galerkin} problem of finding $\tilde{u}^h \in \U^h$ such that
\begin{equation} \label{APGa}
b(\tilde{u}^h;v^h)=f(v^h) \quad (v^h \in   \cT^h(\U^h)).
\end{equation}
Our key requirement on $\V^h$ is that

\begin{equation} 
\label{5}
\gamma^h:=\inf_{0 \neq w^h \in \U^h} \sup_{0 \neq v^h \in \V^h} \frac{b(w^h;v^h)}{\|w^h\|_\U \|v^h\|_\V}\ge \gamma >0,
\end{equation}
holds uniformly in $h$. 
 Then the (projected optimal) test space $\cT^h(\U^h)$ is near-optimal. 
In particular, a generalized C\'{e}a's lemma shows that
\be
\label{GenCea}
\|u-\tilde{u}^h\|_\U \leq    \frac{\|\cB\|_{\cL(\U,\V')}}{\gamma^h}  \inf_{w^h \in \U^h} \|u-w^h\|_\U,
\ee
see e.g. \cite[Thm.~2.1]{75.61}, \cite[Prop.~2.3]{35.8565}, \cite{45.44,58.3}.

Recall that a necessary condition for
realizing our initial objective (i) of linearly scaling computational complexity  is that
\be
\label{desired}
{\rm dim}\,\V^h  \eqsim {\rm dim}\,\U^h,
\ee
uniformly in $h$. 

Note, however, that even when \eqref{desired} holds, determining the corresponding  projected optimal test space still requires solving for each basis function a discrete problem 
which,  generally, has  the same size as the corresponding Petrov-Galerkin problem itself.

Therefore a central objective is to keep also the cost for {\em computing}
$\cT^h(\U^h)$ under control, which is the primary focus of this paper. One strategy  is to {\em localize} the computation of the projected optimal test functions. As advocated by Demkowicz and Gopalakrishnan in several of their works, this localization can be achieved by replacing the ``original'' formulation \eqref{varprob} from the start by a 
 mesh-dependent {\em Discontinuous-Petrov-Galerkin} formulation
\be
\label{DG1} 
b_h(U,v) =   (\cB_hU)(v) = f(v),\quad v\in \V, 
\ee
see e.g. \cite{64.14}. Here, the ``new'' unknown $U$ may now involve in addition to the original field  $u$ also a ``skeleton-component'' that lives on the union 
$\partial\Omega_h$ of cell interfaces of the  underlying mesh  $\Omega_h$. For smooth solutions this skeleton-component agrees with
the traces of $u$ on $\partial\Omega_h$ but these traces may not a priori exist  for all elements in the function space  for $u$.
Choosing now  the (infinite dimensional) test space 
as a ``broken'' space
\be
\label{brokenV}
\V := \prod_{K\in \Omega_h} \V_K, \quad \|v\|^2_{\V}:= \sum_{K\in\Omega_h}\|v\|^2_{\V_K},
\ee
the trial-to-test-mapping $\cT:\U \to \V$ indeed localizes, i.e.,  for $b_h(u,v) = \sum_{K\in\Omega_h}b_K(u,v)$ we have
\be
\label{Tlocal}
\cT u = \sum_{K\in\Omega_h} \cT_K u,\quad \mbox{where}\quad \langle \cT_K u,v\rangle_{\V_K} = b_K(u,v),\quad v\in \V_K.
\ee

One now faces two main issues:
 \begin{itemize}
 \item[(I)]
 Imposing the structure \eqref{brokenV} on the test space, it is not clear that the {\em infinite dimensional} (new) variational formulation \eqref{DG1}
  is well-posed. More precisely, one has to establish {\em uniform} inf-sup stability with respect to a given family of
  partitions $\Omega_h$ with decreasing mesh size parameter $h$.    
 \item[(II)]
 For a given finite dimensional trial space $\U^h$ associated with $\Omega_h$, one still has to find a {\em finite dimensional} test search space
 $$
 \V^h = \prod_{K\in \Omega_h} \V^h_K,
 $$

 that satisfies \eqref{5}.
 \end{itemize}

Regarding our introductory issues (i) and (ii), realizing a linear scaling of the computational work for  the uniformly stable Petrov-Galerkin problems
 one would need to assure that ${\rm dim}\,\V^h \lesssim {\rm dim}\,\U^h$, uniformly in $h$. This would be the case if one were able to assert that for some fixed $M\in \N$,  
 \be
 \label{finiteVK}
 {\rm dim}\,\big(\V_K^h\big) \le M,\quad h \to 0,
 \ee
 suffices to warrant the desired uniform inf-sup stability, and as a consequence, the desired rigorous 
 error-residual relation \eqref{error-res}.
 
 To our knowledge, the only case for which these desiderata have been rigorously established concerns second order elliptic problems
 \cite{75.61}.
   The central objective of this paper is to establish    \eqref{finiteVK} in conjunction with uniform (in $h$) 
inf-sup stability in the DPG context
 for a class of  linear {\em transport equations} with a possibly {\em variable convection} field.

The proof of this result and necessary prerequisits turns out to be quite elaborate. 
Our motivation for investing in a rigorous stability analysis for transport equations stems in part from several envisaged applications 
 that will be addressed in more detail in forthcoming work. This concerns, in particular,   the design and analysis
 of rigorous adaptive methods for transport equations and, in fact, for a somewhat wider scope of problems where 
 transport plays a dominant role such as kinetic models.

\subsection{Layout of the paper}\label{ssec:layout}

In Section~\ref{Stransport} we formulate the first order linear transport equations treated in this  paper.   Section~\ref{Sbroken} is devoted 
to its variational formulation and the proof of its well-posedness, addressing the aforementioned issue (I). 
In Section~\ref{S5} we derive and analyse optimal test functions along with their computable near-optimal counterparts
culminating in the uniform stability  of the DPG scheme, i.e., this section deals with issue (II).

In this work, by $C \lesssim D$ we will mean that $C$ can be bounded by a multiple of $D$, independently of parameters which C and D may depend on. Obviously, $C \gtrsim D$ is defined as $D \lesssim C$, and $C\eqsim D$ as $C\lesssim D$ and $C \gtrsim D$.

\section{Transport equation} \label{Stransport}
For a bounded Lipschitz domain $\Omega \subset \R^n$, let ${\bf b} \in W^0_\infty(\divv;\Omega)$, i.e., ${\bf b} \in L_\infty(\Omega)^n$ with $\divv {\bf b} \in L_\infty(\Omega)$.
We set $$
H({\bf b};\Omega):=\{u \in L_2(\Omega)\colon{\bf b}\cdot \nabla u \in L_2(\Omega)\},
$$
equipped with  the norm $\|u\|_{H({\bf b};\Omega)}^2:=\|u\|_{L_2(\Omega)}^2+\| {\bf b}\cdot \nabla u\|_{L_2(\Omega)}^2$.

In order to define the {\em characteristic, outflow, and inflow} boundary portions $\Gamma_0, \Gamma_+,\Gamma_- \subset \partial\Omega$,
respectively, under the above assumptions on the velocity field $\bbf$ we use the (formal) integration-by-parts formula
$$
\int_\Omega 2 w {\bf b}\cdot \nabla w +w^2 \divv {\bf b} \,d{\bf x} =\int_{\partial\Omega} w^2 {\bf b}\cdot {\bf n} \,d {\bf s},
$$
to define  the characteristic boundary
$\Gamma_{\!0}$ as the largest measurable subset of $\partial\Omega$ such that the left-hand side vanishes for all $w \in H({\bf b};\Omega) \cap C(\bar{\Omega})$ that vanish on $\partial\Omega \setminus \Gamma_{\!0}$.
Similarly, we set the outflow boundary $\Gamma_{\!+}$ as the largest measurable subset of $\partial\Omega \setminus \Gamma_{\!0}$ such that $\int_\Omega 2 w {\bf b}\cdot \nabla w +w^2 \divv {\bf b} \,d{\bf x} \geq 0$ for all $w \in H({\bf b};\Omega) \cap C(\bar{\Omega})$ that vanish on $(\partial\Omega \setminus \Gamma_{\!0})\setminus \Gamma_{\!+}$, and finally, we define the inflow boundary as
$\Gamma_{\!-} = \partial\Omega \setminus (\Gamma_{\!0}\cup\Gamma_{\!+})$.
For continuous ${\bf b}$, it means that $\Gamma_{\!0}:=\{x \in \partial\Omega \colon {\bf b}(x) \cdot {\bf n}(x) =0\}
$ whenever $\bn(x)$ is uniquely defined, and
$
\Gamma_{\!\pm}:=\{x \in \partial\Omega \colon \pm {\bf b}(x) \cdot {\bf n}(x) >0\}
$.

For a ${\bf b} \in W^0_\infty(\divv;\Omega)$, and an $c \in L_\infty(\Omega)$, we consider
the transport equation of finding $u:\Omega \rightarrow \R$ that, for given $f\colon \Omega \rightarrow \R$ and $g \colon \Gamma_{\!-} \rightarrow \R$, solves
\begin{equation} \label{transport}
\left\{
\begin{array}{r@{}c@{}ll}
{\bf b}\cdot \nabla u +c u&\,\,=\,\,& f &\text{ on } \Omega,\\
u&\,\,=\,\, & g &\text{ on } \Gamma_{\!-}.
\end{array}
\right.
\end{equation}

When $g=0$ a first canonical variational formulation of the transport problem reads: find $u$ such that
\be
\label{first}
\int_\Omega  ({\bf b} \cdot \nabla u +c u) v \,d{\bf x}=\int_\Omega  f v\,d{\bf x}
\ee
holds for all smooth test functions $v\in C^\infty(\bar\Omega)$. A second variant  seeks  $u$ such that
\be
\label{second}
\int_\Omega (cv - \divv v {\bf b})u \,d{\bf x}=\int_\Omega  f v -\int_{\Gamma_{\!-}} g v {\bf b}\cdot {\bf n}\,d{\bf x}
\ee
holds for all smooth test functions $v$ that vanish on $\Gamma_{\!+}$. 
Note that in the second formulation, the Dirichlet boundary condition enters as a {\em natural} condition, and therefore this formulation applies equally well for  an inhomogeneous boundary condition on $\Gamma_-$.

Applying Cauchy-Schwarz followed by taking closures, shows that the
 Hilbert spaces
$$
H_{0,\Gamma_{\!\pm}}({\bf b};\Omega):=\clos_{H({\bf b};\Omega)}  \{u \in H({\bf b};\Omega) \cap C(\bar{\Omega}) \colon u=0 \text{ on } \Gamma_{\!\pm}\}.
$$
are relevant for these variational formulations. In fact, 
the operators
$$
\cB:=u \mapsto {\bf b} \cdot \nabla u + cu, \quad \cB^\ast:=v \mapsto  c v-\divv{v {\bf b}} 
$$
are obviously continuous as mappings into $L_2(\Omega)$, i.e.,
$$
\cB \in \cL(H_{0,\Gamma_{\!-}}({\bf b}; \Omega),L_2(\Omega)),\quad \cB^* \in \cL(H_{0,\Gamma_{\!+}}({\bf b}; \Omega),L_2(\Omega)).
$$
In addition, we {\em assume} that
\begin{align} \label{14}
\cB& 
\in \Lis(H_{0,\Gamma_{\!-}}({\bf b}; \Omega),L_2(\Omega)),\\
\cB^\ast&
 \in \Lis(H_{0,\Gamma_{\!+}}({\bf b}; \Omega),L_2(\Omega)), \label{15}
\end{align}
meaning that the first (for $g=0$) or second variational form of the problem is well-posed over $H_{0,\Gamma_{\!-}}({\bf b}; \Omega) \times L_2(\Omega)$ or $L_2(\Omega) \times H_{0,\Gamma_{\!+}}({\bf b}; \Omega)$, respectively.
These assumptions are readily verified for non-zero, constant ${\bf b}$, but are not necessarily satisfied for every vector field ${\bf b}$ as,
for instance,  when flow curves associated to $\pm {\bf b}$ do not reach the boundary. Sufficient conditions for both assumptions are ${\bf b} \in C^1(\bar{\Omega})$ with ${\bf b}(x) \neq 0$ for $x \in \bar{\Omega}$,

or $c-\frac{1}{2} \divv {\bf b} \geq \kappa>0$ a.e. on $\Omega$, for some constant $\kappa$, see \cite[Remark~2.2]{58.3}.

\section{A variational formulation of the transport equation with broken test and trial spaces} \label{Sbroken}
In order to allow us to eventually localize the  determination of the optimal test functions we follow the approach introduced by Demkowicz and Gopalakrishnan \cite{64.14} replacing \eqref{second} by a {\em Discontinuous Galerkin} formulation.
We introduce first the relevant notation.

For any $h$ from an index of mesh parameters, let $\Omega_h$ be a collection of disjoint open Lipschitz domains (`elements') such that $\bar{\Omega}=\bigcup_{K \in \Omega_h} \bar{K}$.
We will refer to such an $\Omega_h$ as a partition of $\Omega$.
For each $K \in \Omega_h$, we split its boundary into characteristic and in- and outflow boundaries, i.e.,  $\partial K=\partial K_0 \cup \partial K_{\!+} \cup \partial K_{\!-}$, and denote by
$$
\partial\Omega_h :=\cup_{K \in \Omega_h} \partial K\setminus \partial K_0
$$
  the {\em mesh skeleton}, i.e., the union of the non-characteristic  boundary portions of the elements.

Let us first assume that  $g=0$  referring to   Remark~\ref{rem1} for $g \neq 0$. Moreover, denoting by $\nabla_h$ the piecewise gradient operator,
 let us introduce  the spaces
$H({\bf b};\Omega_h)=\{v \in L_2(\Omega)\colon{\bf b} \cdot \nabla_h v \in L_2(\Omega)\}$, equipped with squared ``broken'' norm $\|v\|^2_{H({\bf b};\Omega_h)}:=\|v\|_{L_2(\Omega)}^2+\|{\bf b} \cdot \nabla_h v\|_{L_2(\Omega)}^2$,   and let
$$
H_{0,\Gamma_{\!-}}({\bf b};\partial\Omega_h):=\{w|_{\partial\Omega_h}\colon w \in H_{0,\Gamma_{\!-}}({\bf b};\Omega)\},
$$
equipped with quotient norm
\be
\label{quotient}
\|\theta\|_{H_{0,\Gamma_{\!-}}({\bf b};\partial\Omega_h)}:=\inf\{\|w\|_{H({\bf b};\Omega)}\colon\theta =w|_{\partial\Omega_h},\,w \in H_{0,\Gamma_{\!-}}({\bf b};\Omega)\}.
\ee
A standard {\em piecewise} integration-by-parts of the transport equation \eqref{transport}   leads to the following problem:
\begin{equation}
\label{16}
\left\{
\begin{Array}{l}
\text{For } \U:=L_2(\Omega) \times H_{0,\Gamma_{\!-}}({\bf b};\partial\Omega_h),\,\V:= H({\bf b};\Omega_h)

,\\
\text{given } f \in H({\bf b};\Omega_h)', \text{find } (u,\theta) \in \U \text{ such that for all  } v \in \V,\\
b_h(u,\theta;v) := \int_\Omega (cv-{\bf b}\cdot\nabla_h v -v \divv {\bf b})u \,d{\bf x}+\int_{\partial\Omega_h} \llbracket v {\bf b} \rrbracket \theta \,d{\bf s} =f(v).
\end{Array}
\right.
\end{equation}

Here we define as usual for $x \in \partial K \cap \partial K'$, 
$$
\llbracket v {\bf b} \rrbracket(x):=(v {\bf b}|_{K}\cdot{\bf n}_K)(x)+(v {\bf b}|_{K'}\cdot{\bf n}_{K'})(x),
$$
and $\llbracket v \rrbracket(x):=(v {\bf b}|_{K}\cdot{\bf n}_K)(x)$ for $x \in \partial\Omega \cap \partial K$.

The additional independent variable $\theta$ replaces the trace $u|_{\partial\Omega_h}$ which is not defined for general $u \in L_2(\Omega)$. If 
 $f \in L_2(\Omega)$, or, equivalently, $u \in H_{0,\Gamma_{\!-}}({\bf b};\Omega)$, then a reversed integration by parts shows that indeed
$\theta=u|_{\partial\Omega_h}$.

Well-posedness of the variational formulation \eqref{16} is demonstrated in the next theorem.
It is an adaptation of \cite[Thm.~5.1]{35.856} where we employ here slightly different spaces $\U$ and $\V$, and 
where we exhibit explicit bounds on the norms of the operator and its inverse.

In \cite{35.856}, the spaces were chosen such that both $\theta$ and $v$ vanish on $\Gamma_{\!+}$. Also the transport equation here is more general since it may contains a reaction term. For convenience we include the proof.

In the following, we abbreviate $\|\cB^{-1}\|_{\cL(L_2(\Omega),H_{0,\Gamma_{\!-}}({\bf b};\Omega))}$, $\|(\cB^\ast)^{-1}\|_{\cL(L_2(\Omega),H_{0,\Gamma_{\!+}}({\bf b};\Omega))}$, $\|\divv {\bf b}\|_{L_\infty(\Omega)}$, $\|c\|_{L_\infty(\Omega)}$, and $\|c-\divv {\bf b}\|_{L_\infty(\Omega)}$ as $\|\cB^{-1}\|$, $\|B^{-\ast}\|$, $\|\divv {\bf b}\|$, $\|c\|$, and $\|c-\divv {\bf b}\|$ respectively. $\cB, \cB^\ast$, induced by the
conforming formulations \eqref{first}, \eqref{second}, should not be
confused with the operators  $\cB_h$ induced by the DPG formulation.

\begin{theorem} \label{th4}
Assume that $\bbf\in W^0_\infty(\divv;\Omega)$, $c\in L_\infty(\Omega)$ and that 
conditions \eqref{14}, \eqref{15} hold. Then,
defining $\cB_h : \U \to \V'$ by $(\cB_h(u,\theta))(v):=b_h(u,\theta;v)$, 

one has $\cB_h \in \Lis(\U,\V')$ with
\begin{align*}
 \|\cB_h\|_{\cL(\U,\V')} &\leq 2+\|\divv {\bf b}\|+\|c-\divv{\bf b}\|,\\
\|\cB_h^{-1}\|_{\cL(\V',\U)} &\leq \sqrt{\|\cB^\ast\|^2+\tilde{C}_\cB^2},
\end{align*}
where $\tilde{C}_\cB:=(1+\|\cB^{-\ast}\| (1+\|c-\divv {\bf b}\|)) 
\|\cB^{-1}\|(\|c-\divv {\bf b}\|+1)$.
\end{theorem}

\begin{remark} 
As the bilinear form $b_h$ and the operator $\cB_h$, obviously also the spaces $\U$ and $\V$, and the solution $(u,\theta)$ 
depend on $h$,  but we supress these latter dependencies  in the notation.

\end{remark}

\begin{remark}
 A consequence of Theorem~\ref{th4} is that $H({\bf b};\Omega_h) \rightarrow H_{0,\Gamma_{\!-}}({\bf b};\partial\Omega_h)'; v \mapsto \llbracket v {\bf b} \rrbracket$ is surjective.
 \end{remark}
 
Anticipating this latter fact, we can say that the following lemma, which is the first tool
for proving Theorem \ref{th4},  provides an equivalent norm for $H_{0,\Gamma_{\!-}}({\bf b};\partial\Omega_h)'$.
In particular,  it shows that 
$H_{0,\Gamma_{\!-}}({\bf b};\partial\Omega_h)' \simeq H({\bf b};\Omega_h) /H_{0,\Gamma_{\!+}}({\bf b};\Omega)$.

\begin{lemma} \label{lem5}
For $v \in  H({\bf b};\Omega_h)$, one has $\llbracket v {\bf b} \rrbracket \in (H_{0,\Gamma_{\!-}}({\bf b};\partial\Omega_h))'$ with
$$
(\|\cB^{-1}\|(\|c-\divv {\bf b}\|+1)\big)^{-1}
\leq \frac{ \|\llbracket v {\bf b} \rrbracket \|_{H_{0,\Gamma_{\!-}}({\bf b};\partial\Omega_h)'}}{\inf_{z \in H_{0,\Gamma_{\!+}}({\bf b};\Omega)} \|v-z\|_{H({\bf b};\Omega_h)}} \leq 1+\|\divv {\bf b}\|
$$
$(v \in H({\bf b};\Omega_h) \setminus H_{0,\Gamma_{\!+}}({\bf b};\Omega))$. 
\end{lemma}

\begin{proof}
For $v \in H({\bf b};\Omega_h)$, $w \in H_{0,\Gamma_{\!-}}({\bf b};\Omega) \subset H({\bf b};\Omega)$, we have
\begin{equation} \label{170}
\begin{split}
\int_{\partial\Omega_h} \llbracket v {\bf b} \rrbracket w \,d{\bf s} & =\sum_{K \in \Omega_h} \int_K \nabla v \cdot{\bf b} w+v({\bf b}\cdot \nabla w+w\divv{\bf b}) \,d{\bf x}\\
&\leq (1+\|\divv {\bf b}\|) \|v\|_{H({\bf b};\Omega_h)} \|w\|_{H({\bf b};\Omega)},
\end{split}
\end{equation}
showing that $ \|\llbracket v {\bf b}\rrbracket \|_{H_{0,\Gamma_{\!-}}({\bf b};\partial\Omega_h)'} \leq (1+\|\divv {\bf b}\|) \|v\|_{H({\bf b};\Omega_h)}$. Since for $z \in H_{0,\Gamma_{\!+}}({\bf b};\Omega)$ and $w \in H_{0,\Gamma_{\!-}}({\bf b};\Omega)$, $\int_\Omega \nabla z \cdot{\bf b} w+z({\bf b}\cdot \nabla w+w\divv{\bf b})\,d{\bf x}=0$, it follows that $ \|\llbracket z{\bf b}\rrbracket \|_{H_{0,\Gamma_{\!-}}({\bf b};\partial\Omega_h)'}=0$. This shows  that for $v \in H({\bf b};\Omega_h)$, $ \|\llbracket v {\bf b}\rrbracket \|_{H_{0,\Gamma_{\!-}}({\bf b};\partial\Omega_h)'} \leq (1+\|\divv {\bf b}\|)  \inf_{z \in H_{0,\Gamma_{\!+}}({\bf b};\Omega)} \|v-z\|_{H({\bf b};\Omega_h)}$.

To prove the converse estimate let $\divv_h$ denote the piecewise divergence operator.
Given $v \in H({\bf b};\Omega_h)$, let $z \in H_{0,\Gamma_{\!+}}({\bf b};\Omega)$ be the solution of 
$$
\cB^* z =c z-\divv (z{\bf b}) =c v-\divv_h (v{\bf b}),
$$
 whose existence is guaranteed by \eqref{15}. 
 From 
 \begin{equation}
 \label{v-z}
 c(v-z)=\divv_h\big((v-z){\bf b}\big)=(v-z)\divv{\bf b}+{\bf b}\cdot \nabla_h(v-z),
 \end{equation}
  we derive that
\begin{equation} \label{31}
\|{\bf b}\cdot \nabla_h(v-z)\|_{L_2(\Omega)} \leq (\|c-\divv {\bf b}\|) \|v-z\|_{L_2(\Omega)}.
\end{equation}
By \eqref{14}, there exists a $w \in H_{0,\Gamma_{\!-}}({\bf b};\Omega)$ such that $\cB w={\bf b}\cdot \nabla w+c w =v-z$ and
\begin{equation}
\label{wbound}
 \|w\|_{H({\bf b};\Omega)} \leq \|\cB^{-1}\| \|v-z\|_{L_2(\Omega)}. 
 \end{equation}
From the definitions of $w$ and $z$, we have
\begin{align*}
\|v-&z\|^2_{L_2(\Omega)}=\int_\Omega (v-z)({\bf b}\cdot \nabla w +c w) \,d{\bf x}=\sum_{K \in \Omega_h} \int_K (v-z)({\bf b}\cdot \nabla w +c w) \,d{\bf x}\\
&=
\sum_{K \in \Omega_h}  \int_K\big( \divv\big((z-v) {\bf b}\big)+ c(v-z)\big)w\,d{\bf x}+\int_{\partial K} (v-z) w {\bf b}\cdot {\bf n}_K\,d{\bf s}\\
&=\int_{\partial\Omega_h} \llbracket v {\bf b} \rrbracket w\,d{\bf s},
\end{align*}
where we have used \eqref{v-z} in the last step.
Thus, invoking \eqref{wbound}, we have
\begin{align*}
\|v- z\|^2_{L_2(\Omega)} & \le 
   \|  \llbracket v {\bf b}\rrbracket  \|_{H_{0,\Gamma_{\!-}}({\bf b};\partial\Omega_h)'} \|w\|_{H({\bf b};\Omega)}\\
  & \leq \|  \llbracket v {\bf b}\rrbracket  \|_{H_{0,\Gamma_{\!-}}({\bf b};\partial\Omega_h)'} \|\cB^{-1}\| \|v-z\|_{L_2(\Omega)}.
\end{align*}
In other words $\|v-z\|_{L_2(\Omega)} \leq  \|\cB^{-1}\| \|  \llbracket v {\bf b}\rrbracket  \|_{H_{0,\Gamma_{\!-}}({\bf b};\partial\Omega_h)'}$, which, in combination with \eqref{31}, completes the proof. 

\end{proof}

The second tool for the proof of Theorem~\ref{th4} is
the following  well-known consequence of the {\em closed range theorem}.
\begin{lemma} 
\label{lem1}
For reflexive Banach spaces $X$ and $Y$, let $G:X \rightarrow Y'$ be linear. Then $G \in \Lis(X,Y')$ if and only if
\renewcommand{\theenumi}{\roman{enumi}}
\begin{enumerate} 
\item \label{I} $G \in \cL(X,Y')$,
\item \label{II}$ \beta:=\inf_{0 \neq y \in Y}\sup_{0 \neq x \in X} \frac{(G x)(y)}{\|x\|_X \|y\|_Y}>0$,
\item \label{III} $\forall 0 \neq x \in X$, $\exists y \in Y$, with $(Gx)(y) \neq 0$.
\end{enumerate}
Moreover, one has  $\|G^{-1}\|_{\cL(Y',X)} =\frac{1}{\beta}$.
\end{lemma}
\noindent Since $G \in \Lis(X,Y')$  is equivalent to $G'\in \Lis(X',Y)$, the roles of $X$ and $Y$ in \eqref{II} and \eqref{III} can be interchanged.

\begin{proof}[Proof of Theorem~\ref{th4}] The bound on $\|\cB_h\|_{\cL(U,V')}$ follows easily from \eqref{170}.

We will establish  the remaining claim with the aid of Lemma \ref{lem1}. To verify first \eqref{III},
  let $(u,\theta) \in \U$ be such that $b_h(u,\theta;v)=0$ for all $v \in H({\bf b};\Omega_h)$. Considering first all  $v$ from the subspace $H_{0,\Gamma_{\!+}}({\bf b};\Omega)$, \eqref{15} yields $u=0$ because $\cB$ agrees with $\cB_h$ on this subspace. 
  By considering now for any  $K \in \Omega_h$ all $v$ with ${\rm supp}\,v \subset K$, we infer that $\theta|_{\partial K }=0$, and so $\theta=0$.

Finally, let 
$v \in H({\bf b};\Omega_h)$
be given. By \eqref{15}, there exists a $v_1=v_1(v) \in H_{0,\Gamma_{\!+}}({\bf b};\Omega)$ with 
\be
\label{equal}
c v_1-\divv (v_1 {\bf b})=c v -\divv_h (v {\bf b}),\quad  \|v_1\|_{H({\bf b};\Omega)} \leq \|\cB^{-\ast}\|\|c v-\divv_h (v {\bf b})\|_{L_2(\Omega)}
\ee
Thus
$
\|v_1\|_{H({\bf b};\Omega)} \leq  \|\cB^{-\ast}\| (1+\|c-\divv {\bf b}\|)
\|v\|_{H({\bf b};\Omega_h)},
$
which says 
\begin{equation}
\label{vz}
\|v_1-v\|_{H({\bf b};\Omega_h)} \leq (1+\|\cB^{-\ast}\| (1+\|c-\divv {\bf b}\|))
\|v\|_{H({\bf b};\Omega_h)}.
\end{equation}
 Moreover, we have $v_1=v$ when $v \in H_{0,\Gamma_{\!+}}({\bf b};\Omega)$,  so that
for any $z\in H_{0,\Gamma_{\!+}}({\bf b};\Omega)$ we have $v_1(v-z) -(v-z)= v_1(v)- v$ so that \eqref{vz} actually  gives 
\begin{equation} \label{32}
\begin{split}
\|v_1-v\|_{H({\bf b};\Omega_h)} & \leq (1+\|\cB^{-\ast}\| (1+\|c-\divv {\bf b}\|)) \inf_{z \in H_{0,\Gamma_{\!+}}({\bf b};\Omega)} \|v-z\|_{H({\bf b};\Omega_h)}\\
& \leq \tilde{C}_\cB 
\|\llbracket v {\bf b} \rrbracket \|_{H_{0,\Gamma_{\!-}}({\bf b};\partial\Omega_h)'}
\end{split}
\end{equation}
by an application of Lemma~\ref{lem5}.

There exists a $\theta \in H_{0,\Gamma_{\!-}}({\bf b};\partial\Omega_h)$ with
$
\|\llbracket v {\bf b} \rrbracket \|_{H_{0,\Gamma_{\!-}}({\bf b};\partial\Omega_h)'} =
\frac{\int_{\partial\Omega_h} \llbracket v {\bf b} \rrbracket \theta \,d{\bf s}}{\|\theta\|_{H_{0,\Gamma_{\!-}}({\bf b};\partial\Omega_h)} }.
$
By selecting $\|\theta\|_{H_{0,\Gamma_{\!-}}({\bf b};\partial\Omega_h)} =\tilde{C}_\cB^{-1} \|v_1-v\|_{H({\bf b};\Omega_h)}$, and invoking \eqref{32}, we have
\begin{equation} \label{33}
\|\theta\|^2_{H_{0,\Gamma_{\!-}}({\bf b};\partial\Omega_h)} =\tilde{C}_\cB^{-2} \|v_1-v\|_{H({\bf b};\Omega_h)}^2
\leq \int_{\partial\Omega_h} \llbracket v {\bf b} \rrbracket \theta \,d{\bf s}.
\end{equation}

Similarly, there exists a $u \in L_2(\Omega)$ with $\|\cB^\ast v_1\|_{L_2(\Omega)}=\frac{\int_{\Omega} c u v_1-u \divv(v_1 {\bf b})\,d{\bf x}}{\|u\|_{L_2(\Omega)}}$. By selecting $\|u\|_{L_2(\Omega)}=\|\cB^{-\ast}\|^{-1}\|v_1\|_{H({\bf b};\Omega)}$, and using the first relation in \eqref{equal}, we infer that
\begin{equation} \label{34}
\|u\|^2_{L_2(\Omega)}=\|\cB^{-\ast}\|^{-2}\|v_1\|^2_{H({\bf b};\Omega)} \leq \int_{\Omega} c u v-u \divv_h(v {\bf b})\,d{\bf x}.
\end{equation}

The combination of \eqref{33} and \eqref{34} shows that
\begin{align*}
&(\|\cB^{-\ast}\|^2+\tilde{C}_\cB^2)^{-\frac{1}{2}} \|v\|_{H({\bf b};\Omega_h)} \\
& \leq
(\|\cB^{-\ast}\|^2+\tilde{C}_\cB^2)^{-\frac{1}{2}} (\|v_1\|_{H({\bf b};\Omega)}+ \|v_1-v\|_{H({\bf b};\Omega_h)})\\ 
& \leq
\sqrt{\|\cB^{-\ast}\|^{-2} \|v_1\|_{H({\bf b};\Omega)}^2+\tilde{C}_\cB^{-2} \|v_1-v\|_{H({\bf b};\Omega_h)}^2} \\
& =
\frac{\|\cB^{-\ast}\|^{-2} \|v_1\|_{H({\bf b};\Omega)}^2+\tilde{C}_\cB^{-2} \|v_1-v\|_{H({\bf b};\Omega_h)}^2}{
\sqrt{\|u\|^2_{L_2(\Omega)}+\|\theta\|^2_{H_{0,\Gamma_{\!-}}({\bf b};\partial\Omega_h)}}}\\
& \leq 
\frac{b(u,\theta;v)}{
\sqrt{\|u\|^2_{L_2(\Omega)}+\|\theta\|^2_{H_{0,\Gamma_{\!-}}({\bf b};\partial\Omega_h)}}}.
\end{align*}
Invoking Lemma \ref{lem1} completes the proof.
\end{proof}

\begin{remark}[Inhomogenous boundary condition] \label{rem1} The variational formulation \eqref{16} is not suited for an inhomogeneous boundary condition $u=g$ on $\Gamma_{\!-}$, because the homogeneous condition $u=0$ on $\Gamma_{\!-}$ has been incorporated in the space $H_{0,\Gamma_{\!-}}({\bf b};\partial\Omega_h)$ for the variable $\theta$.

Therefore, for $g \neq 0$, let $\bar{g} \in H({\bf b},\Omega)$ be an extension of $g$. Then with $\bar{u}:=u-\bar{g}$, one may apply the variational formulation \eqref{16} to the transport equation 
$$
\left\{
\begin{array}{r@{}c@{}ll}
{\bf b}\cdot \nabla \bar{u} +c \bar{u}&\,\,=\,\,& f -{\bf b}\cdot \nabla \bar{g}-c\bar{g} &\text{ on } \Omega,\\
\bar{u}&\,\,=\,\, & 0 &\text{ on } \Gamma_{\!-},
\end{array}
\right.
$$
which gives the problem of finding $(\bar{u},\bar{\theta}) \in \U$ such that for all $v \in \V$,
\begin{align*}
b_h(\bar{u},\bar{\theta};v)&=f(v)-\int_\Omega({\bf b}\cdot \nabla \bar{g}+c\bar{g})v\,d{\bf x}\\
&=f(v)+\int_\Omega ({\bf b}\cdot \nabla_h v+v \divv {\bf b}-cv)\bar{g}\,d{\bf x}-\int_{\partial\Omega_h} \llbracket v {\bf b} \rrbracket  \bar{g}\,d{\bf s} .
\end{align*}
When $f \in L_2(\Omega)$, it holds that $\bar{\theta}=\bar{u}|_{\partial\Omega_h}=(u-\bar{g})|_{\partial\Omega_h}$.

Alternatively, using that only the space for $\theta$ is inappropriate for $g \neq 0$, by subtracting $\int_{\Omega_h} \llbracket v {\bf b} \rrbracket  \bar{g}\,d{\bf s}$ from both sides of \eqref{16}, and introducing $\bar{\theta}:=\theta-\bar{g}|_{\partial\Omega_h}$, one arrives at the problem of finding $(u,\bar{\theta}) \in \U$ such that for all $v \in \V$,
$$
b_h(u,\bar{\theta};v)=f(v)-\int_{\partial\Omega_h} \llbracket v {\bf b} \rrbracket  \bar{g}\,d{\bf s}.
$$
\end{remark}

\section{Optimal test functions} \label{S5}
\newcommand{\T}{\mathbb{T}}

\subsection{Preliminary remarks and a roadmap}\label{ssec:roadmap}
Given a family of finite dimensional piecewise polynomial trial spaces $\U^h\subset \U= L_2(\Omega)\times
H_{0,\Gamma_-}(\bbf;\partial\Omega_h)$, parametrized by the mesh size parameter $h$, we wish to construct a uniformly stable 
finite dimensional family
of  test search spaces
$\V^h\subset \V =H({\bf b};\Omega_h)$ which, due to the product structure of $\V$, have the form
 $$
\V^h =   \prod_{K\in \Omega_h} \V_K.
$$
 By uniformly stable we mean of course that there exists a positive constant $\gamma >0$ such that \eqref{5} holds for the present setting, i.e.,

\be
\label{uniformstab}
\inf_{(u,\theta)\in\U^h}\sup_{v\in \V^h}\frac{b_{ h}(u,\theta;v)}{\|(u,\theta)\|_\U \|v\|_\V}=: \gamma^h \ge \gamma,\quad (h>0).
\ee

In view of \eqref{quotient}, it suffices to establish inf-sup stability for a slightly modified formulation replacing the 
component $\theta \in H_{0,\Gamma_-}(\bbf;\partial\Omega_h)$ by a  suitable ``lifting'' $w\in H_{0,\Gamma_-}(\bbf;\Omega)$,  i.e., $w|_{\partial\Omega_h}=\theta$, which we express by writing
$$
b_h(u,w;v) :=  \sum_{K\in\Omega_h}b_K(u,w;v),
$$
where
\begin{equation}\label{bhw}
\begin{split} 
b_K(u,w;v) &:=  \int_K (cv- \divv (\bbf v))u d\bx + \int_{\partial K}\bbf\cdot \bn_K vw d\bs\\ 
& =  \int_K (c- \divv \bbf)vu + (w-u)\bbf\cdot\nabla v+  v\bbf\cdot\nabla w +vw \divv \bbf d\bx.
\end{split}
\end{equation}
In consequence we endow $\U$ with the norm
\be
\label{Uw}
\|(u,w)\|_\U^2:= \|u\|^2_{L_2(\Omega)} + \|w\|_{H(\bbf;\Omega)}^2,
\ee
and recall from \eqref{Tlocal} that
  the  trial-to-test map $\cT: \U \to \V$  has now also product form
$$
\cT(u,w) = \big(\cT_K(u|_K,w|_{\partial K})\big)_{K \in \Omega_h},
$$ 
where each local optimal test-function $t_K= t_K(u,w) := \cT_K(u|_K,w|_{\partial K})$ is defined by
\be
\label{TK}
\langle t_K,v \rangle_{H({\bf b};K)} = b_K(u,w;v)\quad (v\in H(\bbf;K)).
\ee

Our goal is to identify stable formulations for variable fields $\bbf$ subject to the
assumptions made earlier. For such general fields one cannot expect to find truly optimal test functions, but essentially we will be able to do so for {\em piecewise constant} fields.
Therefore we will introduce a perturbed bilinear form
\be \label{barbhw0}
\breve b_h(u,w;v) :=  \sum_{K\in \Omega_h}\breve b_K(u,w;v),
\ee
where the summands $\breve b_K(u,w;v)$ are defined as follows. 
Suppose that $\cc_K,\bb_K,\dd_K$ are approximations  on $K$ to the fields $c- \divv \bbf,\bbf, \divv \bbf$, respectively. 
Then in accordance to \eqref{bhw}, we set
\begin{equation} \label{200}
\begin{split}
\breve b_K(u,w;v) :=&\int_K \cc_K vu + (w-u)\bb_K \cdot\nabla v+  v\bb_K \cdot\nabla w +\dd_K vw d\bx \\
=&\int_K \big((\cc_K +\divv \bb_K)v  -\divv(\bb_K v)\big)u+(\dd_K-\divv \bb_K) w v \,d\bx\\
&+\int_{\partial_K} \bb_K \cdot {\bf n}_K v w d\bs .
\end{split}
\end{equation}
These approximations will be specified later in Sect.~\ref{the_main_result}.
Its effect is that, for $\dd_K\neq \divv \bb_K$,  the corresponding   (near) optimal test functions no longer depend only on the traces $w|_{\partial \Omega_h}$. 

Given such a perturbed form $\breve b_h$ and a finite dimensional (piecewise polynomial) trial space $\U^h\subset \U$, we then have to carry out two main tasks:

(i) 
for any $(u,w)\in \U^h$ we wish to find a $\breve t= \breve t(u,w;\breve b_h) \in \V$, preferably piecewise polynomial,
such that $\breve b_h(u,w;\breve t) \gtrsim \|(u,w)\|_\U \|\breve t\|_{\V} $, of course, uniformly in $h$ and  in $(u,w) \in \U^h$.

(ii)
Starting from the simple decomposition
\be
\label{split}
b_h(u,w;\breve t)= \breve b_h(u,w;\breve t)+ (b_h(u,w;\breve t) - \breve b_h(u,w;\breve t)),
\ee
the choice  of $\breve t$ allows us to handle the first summand. It then remains to show for the second summand that
\be
\label{wish}
|b_h(u,w;\breve t) - \breve b_h(u,w;\breve t)|  \le  \delta \|(u,w)\|_\U \|\breve t\|_{\V},
\ee
holds for a sufficiently small $\delta >0$ , depending on the inf-sup constant for the first summand. 

Note that after having established (i)-(ii), any test search space
$$
\V^h \supseteq \Span\{\breve t(u,w;\breve b_h)\colon (u,w) \in \U^h\}
$$
will be uniformly stable in the sense of \eqref{uniformstab}.

Concerning (i), in Sect.~\ref{S_piecewise_constant} we will see that after equipping the test space by a different but equivalent norm, the trial-to-test map can be evaluated exactly.
It turns out, however, that the resulting truly optimal test functions corresponding to $\breve b_h$ are possibly very sensitive to perturbations in the convection field.
Therefore, in order to be able to simultaneously establish (ii), we will have to replace them by near optimal test functions. 

Another issue we will have to deal with is the following: If one has a bilinear form for which the corresponding operator, in the infinite dimensional setting, is boundedly invertible, then for given finite dimensional trial space, the corresponding optimal test space gives an inf-sup stable pair.
The convection field corresponding to 
the perturbed bilinear form $\breve{b}_h$, however, will generally not be in $W^0_\infty(\divv;\Omega)$, and so the theory about well-posedness in the infinite dimensional setting developed in Sect.~\ref{Sbroken} will not be applicable to this perturbed form. 
We will establish the inf-sup stability needed in (i) partly by direct calculations, and partly by invoking the well-posedness of the original bilinear form.

\subsection{Reduction to two-point boundary value problems} 

From \eqref{bhw}--\eqref{TK} recall the local variational problems
\be \label{201}
\langle t_K, v\rangle_{H(\bbf;K)} = \int_K (c v- \divv (\bbf v))u  + d wv\,d\bx+ \int_{\partial K}\bbf\cdot \bn_K vw d\bs
\ee
that determine the local optimal test functions $t_K=t_K(u,w)$. Compared to \eqref{bhw}, here we consider an ``extended'' form including a term $dwv$, similarly to \eqref{200},
because we will specify this below to approximations $\breve{b}_h$ to $b_h$.
  
When $u|_K \in H({\bf b};K)$, as is the case when $u$ is piecewise polynomial,
we can reverse integration by parts in these local problems, which reveals that they  have the  following strong form  
 \be
 \label{tK}
\left\{
\begin{array}{rcll}
-[\partial_{\bf b}^2 t_K +\partial_{\bf b} t_K\divv {\bf b} ]+t_K & = & \partial_{\bf b} u + c u +d w & \text{in } K,\\
\partial_{\bf b} t_K & = &  w-u & \text{on } \partial K_{\!+} \cup \partial K_{\!-},
\end{array}
\right.
\ee
Using a transformation to characteristic coordinates defined by 
$$
\frac{d}{d\lambda}\chi(\lambda,\bs) = \bbf (\chi(\lambda, \bs)),\quad  \chi(0,\bs)=\bs \in K_{\!-},
$$
\eqref{tK} can be viewed as a family of ordinary two-point boundary value problems. In fact, defining
  $\hat{t}_K:=t_K \circ \chi$,
 
 $\hat{u}:=u \circ \chi$, 
 
 and denoting
by $L(s)>0$ the smallest number for which $\chi(L(\bs),\bs) \in K_{\!+}$, \eqref{tK} takes the form
$$
\begin{array}{rcll}
-[\frac{d^2 \hat{t}_K}{d \lambda^2} +\frac{d \hat{t}_K}{d \lambda}  (\divv {\bf b}) \circ \chi ]+\breve{t}_K & = & \frac{d \hat{u}}{d \lambda} + c \circ \chi\, \hat{u} + (d w) \circ \chi& \text{in } (0,L(\bs)),\\
\frac{d \hat{t}_K}{d \lambda} & = &  w \circ \chi -\hat{u} & \text{at } \{0,L(\bs)\}, \end{array}\!\!\!\!\!\!\!\!
$$
which, in principle, we can   solve for each $\bs$ at any desired accuracy and, for certain $\bb, u, w$ even exactly .
 
\subsection{(Piecewise) constant convection field} \label{S_piecewise_constant}
A simple explicit representation of $t_K$ can be obtained  when 
  ${\bf b}|_K=\bb$ is constant, $K$ is polyhedral, and the restrictions of   $u$, $w$, $c$ and $d$  to each $K$ are polynomial. The
  characteristics are then straight lines and the local optimal test function $t_K$, can then be determined analytically.  It fails, however, to be itself a piecewise    polynomial.
In order to arrive in this case at (piecewise) polynomial local optimal test functions we follow an idea from \cite{64.14}.
Namely, we equip $H(\bb;K)$ with an  alternative, but equivalent Hilbertian norm.
The key is the following simple lemma.

\begin{lemma} \label{lem2} 
For $k \geq h>0$, it holds that
$$
k^2 \| v'\|_{L_2(0,h)}^2+\|v\|_{L_2(0,h)}^2 \eqsim k^2 \| v'\|_{L_2(0,h)}^2+h |v(0)|^2 \quad(v \in H^1(0,h)),
$$
where the  (hidden) constants are independent of $h$, $k\ge h$, and $v$.
\end{lemma}

\begin{proof} First note that it is sufficient to prove the result for $k=h>0$. Next, a homogeneity argument shows that it is sufficient to consider the case that $h=1$. For this case, the statement follows from $\|v\|_{L_2(0,1)} \leq \|v-v(0)\|_{L_2(0,1)}+|v(0)| \lesssim \|v'\|_{L_2(0,1)}+|v(0)|$ by Friedrich's inequality, together with $|v(0)| \lesssim \|v\|_{H^1(0,1)}$ by Sobolev's inequality.
\end{proof}

\begin{remark} Obviously, the condition $k \geq h$ can be replaced by $k \geq C h$ for some $C>0$. Since this constant would then propagate through essentially all subsequent developments combined with further unspecified constants, we keep  for convenience  $C=1$.
\end{remark}

\begin{proposition} 
\label{prop2} 
Let $K \subset \R^n$ be a Lipschitz domain, 

and assume that 
  $0 \neq \bb \in \R^n$ is a {\em constant}.
Denoting by 
$r({\bf s})$   the distance from ${\bf s} \in \partial K_{\!-}$ to $\partial K_{\!+}$ along $\bb$,
one has for $q_K \geq |\bb|^{-1} \diam(K)$,
\begin{align*}
&
q_K^2 \|\partial_\bb v\|^2_{L_2(K)} + \|v\|^2_{L_2(K)} \ \\
& \eqsim 
q_K^2 \|\partial_\bb v\|^2_{L_2(K)}+\int_{\partial K_{\!-}} |v({\bf s})|^2 |({\textstyle \frac{\bb}{|\bb|}} \cdot {\bf n}_K) ({\bf s})| r({\bf s})d{\bf s},
\quad(v \in H(\bb;K)),
\end{align*}
where the constants are those from Lemma \ref{lem2}.
\end{proposition}

\begin{proof} 
Obviously, it is sufficient to prove the statement for $q_K =|\bb|^{-1} \diam(K)$, so that $q_K^2 \|\partial_\bb v\|^2_{L_2(K)}=\diam(K)^2\|\partial_{\bb/|\bb|} v\|^2_{L_2(K)}$.
Without loss of generality we may consider the case that $\bb/|\bb| ={\bf e}_1$.
Given $x_2,\ldots,x_n$, let ${\bf s}$ denote the projection of $(x_2,\ldots,x_n)$ on $\partial K_{\!-}$ along the $x_1$-direction.
We apply Lemma~\ref{lem2} for the integration in the $x_1$-direction, where we use that for each $\bs$ the quantity $r(\bs)$ plays the role
of $h$ in Lemma \ref{lem2} while 
$\diam(K) \geq r({\bf s})$ plays the role of $k$ in Lemma \ref{lem2}.
Integrating the result over $x_2,\ldots,x_n$ and using that $d {\bf s} =\frac{|\bb|}{|\bb\cdot {\bf n}_K({\bf s})|}d x_2 \ldots d x_n$, confirms 
the claim.
\end{proof}

\begin{remark} Proposition~\ref{prop2}  can be generalized to non-constant ${\bf b}$ by applying the coordinate transformation $\chi$ involving the characteristic coordinates. The constants absorbed by the equivalence symbol $\eqsim$   depend  then also on the Jacobian of $\chi$, and the length of the characteristic curve sections connecting the in- and outflow boundary, see also \cite{64.573}.
\end{remark}

The above lines of thought were already used in \cite[(3.22)]{64.14} where, however, the (necessary) factor $|(\frac{\bb}{|\bb|} \cdot {\bf n}_K) ({\bf s})| r({\bf s})$ in the integrand of the integral over $\partial K_{\!-}$ is missing.

For later use we record the following consequence of Proposition~\ref{prop2}. 
\begin{remark}
\label{rem:4.3}  
For a {\em constant} $\bb\neq 0$, and
$$
\diam(K) \leq |\bb|,
$$
the scalar product
$$
\langle\!\langle v, z\rangle\!\rangle_{K,\bb}:=\langle \partial_\bb v,\partial_\bb z\rangle_{L_2(K)}+\int_{\partial K_{\!-}} v({\bf s}) z({\bf s}) |({\textstyle \frac{\bb}{|\bb|}}\cdot {\bf n}_K) ({\bf s})| r({\bf s})d{\bf s}.
$$
gives rise to an {\em equivalent} norm on $H(\bb;K)$,
so that this scalar product can be used to determine the local optimal test function.
\end{remark} 

Assuming that $u|_K \in H(\bb;K)$, the local optimal test function $t_K=\cT_K(u|_K,w|_{K})$ that results from replacing $\langle\,,\,\rangle_{H(\bb;K)}$ by 
$\langle\!\langle \,, \,\rangle\!\rangle_{K,\bb}$ in \eqref{201}, is the solution of
 \begin{equation} \label{18}
\left\{
\begin{array}{rcll}
-\partial_\bb^2 t_K & = & \partial_\bb u + c u  + d w& \text{on } K,\\
\partial_\bb t_K + t_K \frac{|\bb \cdot {\bf n}_K|}{|\bb| \bb \cdot {\bf n}_K} r & = &  w-u & \text{on } \partial K_{\!-},\\
\partial_\bb t_K & = &  w-u & \text{on } \partial K_{\!+}.
\end{array}
\right.
\end{equation}

We consider the case of $K$ being {\em convex}. By a rotation of coordinates, for solving \eqref{18} it is sufficient to consider the case of
$$
\bb=|\bb| {\bf e}_1,
$$
or, equivalently, to read $(x_1,\ldots,x_n)$ as Cartesian coordinates with the first basis vector being equal to $\bb / |\bb|$.
For ${\bf x}=(x,{\bf y}) \in K$, let $x_{\!\pm}({\bf y})$ be such that $(x_{\!\pm}({\bf y}),{\bf y})  \in \partial K_{\!\pm}$, see Figure~\ref{fig1}. 
\begin{figure}
\begin{center}
  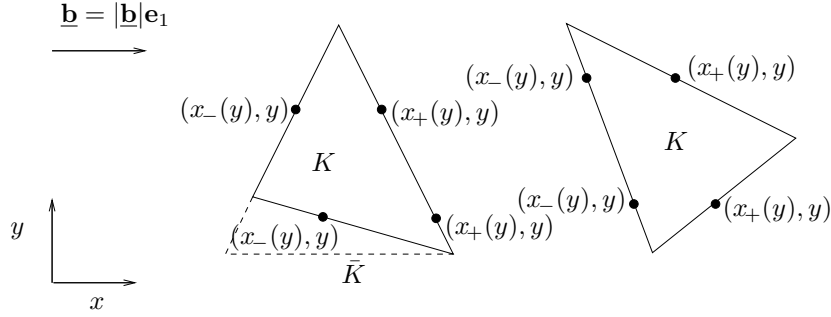
\end{center}
\caption{$x_{\!\pm}$ on a triangle $K$ with two (left) or one (right) inflow boundaries (the extended triangle $\bar K$ will get its meaning in Sect.~\ref{Sstab})}
\label{fig1}
\end{figure}

Furthermore, although not essential, we will think of $c$ and $d$ as being constant on $K$ as well, and write them as $\cc$ and $\dd$.
 Then the solution
 $$
 t_K=\cT_{K,\bb,\cc,\dd}(u|_K,w|_{K})
 $$
 of \eqref{18} is given by
\begin{equation} 
\label{optimaltest}
\begin{split}
& |\bb|\, t_K(x,{\bf y})=- |\bb|^{-1} \int_{x_{\!-}({\bf y})}^x \int_{x_{\!-}({\bf y})}^z (\partial_\bb u + \cc u + \dd w)(q,{\bf y}) d q
dz
\\
&+ \Big(w(x_{\!+}({\bf y}),{\bf y})-u(x_{\!-}({\bf y}),{\bf y})+|\bb|^{-1} \int_{x_{\!-}({\bf y})}^{x_{\!+}({\bf y})} (\cc u + \dd w)(q,{\bf y}) dq \Big)\Big(x-x_{\!-}({\bf y})\Big)\\
&+ |\bb|^2 \frac{w(x_{\!+}({\bf y}),{\bf y})-w(x_{\!-}({\bf y}),{\bf y})+|\bb|^{-1} \int_{x_{\!-}({\bf y})}^{x_{\!+}({\bf y})} (\cc u+ \dd w) (q,{\bf y}) dq}{x_{\!+}({\bf y})-x_{\!-}({\bf y})}.
\end{split}
\end{equation}

For $K$ being polyhedral, the function ${\bf y} \mapsto x_{\!\pm}({\bf y})$ is continuous {\em piecewise} linear. Using that for any univariate polynomial $p$ of degree $m \geq 1$, $(\alpha,\beta) \mapsto \frac{p(\beta)-p(\alpha)}{\beta - \alpha}$ is a bivariate polynomial of degree $m-1$, we infer that for $u$, $w$  being polynomials on $K$, $t_K$ is a continuous piecewise polynomial on $K$.

\subsection{A stability issue} \label{Sstab}

Unfortunately, depending on the angle between $\bb$ and a face,  the derivatives of $x_{\!+}$ or $x_{\!-}$ can be arbitrarily large.
Consequently, generally the problem of determining an optimal test function is not stable regarding its dependence on $\bb$.

 Consequently, a serious impediment arises when the piecewise constant $\bb$ is an approximation to a variable field $\bbf$.
As will be seen later (last statement of Lemma~\ref{lem4}), when treating the second summand in \eqref{split},  one eventually has to control the $H^1$-norm of the
test functions via inverse estimates which requires controling the derivatives of $x_\pm({\bf y})$.

To tackle this problem, for $K$ being an {\em $n$-simplex},
we define an approximation $\breve{t}_K$ to $t_K$ by discarding higher order terms, which is stable as a function of $\bb$.
Moreover, whereas, for polynomial $u$ and $w$, $t_K$ is only {\em piecewise} polynomial w.r.t. a partition of $K$ that depends on the field $\bb$, $\breve{t}_K$ will be {\em polynomial}.

To define $\breve{t}_K$, first we construct a {\em polyhedral} set  $\bar{K}$ that contains $K$ as follows. The number of inflow faces of $K$ is between $1$ and $n-1$. Let $F$ be the inflow face whose normal makes the smallest angle with  $\bb$, and let $v$ denote the vertex of $K$ that does not belong to $F$. Finally let $H_F$ denote the $(n-1)$-hyperplane containing $F$. The ``shadow'' of $K$ on $H_F$, i.e.,
$$
\bar F := \big\{\bx\in H_F: \{\bx + t\bb: t\in \R\} \cap K \neq \emptyset\big\}
$$
is an $(n-1)$-dimensional polyhedron containing $F$ and let $\bar K$ denote the convex hull  of $v$ and $\bar F$,  cf. Figure~\ref{fig1} for $n=2$.
Then, by construction, $\bar{K}$ has only one inflow face $\bar{K}_{-} := \bar F$, and 
$K \subseteq \bar{K}$ with equality if and only if $K$ has only one inflow face, namely $K_- =F$.

For ${\bf x}=(x,{\bf y}) \in \bar{K} \supset K$, let $\bar{x}_{\!-}({\bf y})$ be the linear function with $(\bar{x}_{\!-}({\bf y}),{\bf y})  \in \partial \bar{K}_{-}$,
i.e., $\bar x(\by)$ agrees with $x(\by)$ on $F$.
Then we have
\begin{align}
&\diam (\bar{K}) \lesssim \diam(K),\label{108}\\
&|\bar{x}_{\!-}|_{W^1_\infty(\bar{K})} \lesssim 1, \label{104}
\end{align}
where both constants  depend only on (an upper bound for) the shape regularity parameter
$$
\varrho_K:=\frac{\diam(K)}{\sup\{\diam(B)\colon B \text{ a ball in }K\}}.
$$

For polynomials  $u$ and $w$  on $K$, say of degree $m$, we define now the {\em local   test function}
$$
\breve{t}_K= \breve{t}_{K,\bb,\cc,\dd}(u|_K,w|_K) \in \cP_{m+1}
$$
by
\begin{equation} \label{near-opt}
\begin{split}
 |\bb|\, \breve{t}_K(x,{\bf y}):=&\Big(w(\bar{x}_{\!-}({\bf y}),{\bf y})-u(\bar{x}_{\!-}({\bf y}),{\bf y}) \Big)\Big(x-\bar{x}_{\!-}({\bf y})\Big)\\
&+|\bb|(\partial_{\bb} w(\bar{x}_{\!-}({\bf y}),{\bf y}) + \cc u(\bar{x}_{\!-}({\bf y}),{\bf y})+\dd w(\bar{x}_{\!-}({\bf y}),{\bf y})).
\end{split}
\end{equation}
Since $u$ and $w$ are uniquely defined as polynomials on all of $\R^n$, the polynomial $\breve t_K$ is well-defined outside $K$ and in particular on $\bar K$.

We will show that $\breve t_K$ deserves to be termed {\em near-optimal local test function} and as a first step
we quantify the effect of the above modification.
\begin{lemma} 
\label{lem3}
Let $u|_K$ and $w|_K$ be polynomials of degree $m$. Then 
$$
\|t_K-\breve{t}_K\|_{H(\bb;K)} \lesssim |\bb|^{-1} \diam(K)\big[\|u\|_{L_2(K)}+\|w\|_{H(\bb;K)}+\|\partial_\bb u\|_{L_2(K)}+\|\partial_\bb^2 w\|_{L_2(K)}\big],
$$
only dependent on upper bounds for $m$, $|\cc|$, $|\dd|$ and $\varrho_K$, and, as always, assuming that $\diam(K) \leq |\bb|$.
\end{lemma}

\begin{proof} In view of the definitions of $t_K$ and $\breve{t}_K$, we split their difference, as well as the difference of $\partial_\bb t_K$ and $\partial_\bb\breve{t}_K$, into a number of terms whose $L_2(K)$-norms we bound in a straightforward way.
We start with the first task.
It holds that
\begin{align*}
&\|(x,{\bf y})\mapsto |\bb|^{-2} \int_{x_{\!-}({\bf y})}^x \int_{x_{\!-}({\bf y})}^z (\partial_\bb u + \cc u + \dd w)(q,{\bf y}) d q dz\|_{L_2(K)}
\\ &\hspace*{5em}\lesssim |\bb|^{-2} \diam(K)^2 \|\partial_\bb u + \cc u + \dd w\|_{L_2(K)},
\intertext{and}
&\|(x,{\bf y})\mapsto |\bb|^{-2} (x-x_{\!-}({\bf y}))\int_{x_{\!-}({\bf y})}^{x_{\!+}({\bf y})} (\cc u + \dd w)(q,{\bf y}) dq \|_{L_2(K)}
\\ &\hspace*{5em}\lesssim |\bb|^{-2} \diam(K)^2 \| \cc u + \dd w\|_{L_2(K)}.
\end{align*}
Writing
\begin{align*}
& |\bb|^{-1}\big\{
(w(x_{\!+}({\bf y}),{\bf y})-u(x_{\!-}({\bf y}),{\bf y}))(x-x_{\!-}({\bf y}))\\&
\hspace*{10em}-
(w(\bar{x}_{\!-}({\bf y}),{\bf y})-u(\bar{x}_{\!-}({\bf y}),{\bf y}))(x-\bar{x}_{\!-}({\bf y}))
\big\}
=\\& 
|\bb|^{-1} \big(w(x,{\bf y})-u(x,{\bf y})\big) \big(\bar{x}_{\!-}({\bf y})-x_{\!-}({\bf y})\big)\\
&+
|\bb|^{-2} \Big(\int_{x}^{x_+({\bf y})} \partial_\bb w(z,{\bf y})dz + \int_{x_{\!-}({\bf y})}^{x} \partial_\bb u(z,{\bf y})dy \Big) \big(\bar{x}_{\!-}({\bf y})-x_{\!-}({\bf y})\big)
\\
&+
|\bb|^{-2}\Big(\int_{\bar{x}_{\!-}({\bf y})}^{x_{\!+}({\bf y})} \partial_{\bb} w(z,{\bf y}) \,dz + \int_{x_{\!-}({\bf y})}^{\bar{x}_{\!-}({\bf y})} \partial_{\bb} u(z,{\bf y}) \,dz\Big)\big(x-\bar{x}_{\!-}({\bf y})\big),
\end{align*}
the $L_2(K)$-norm of the expression on the first line at the right-hand side can be bounded by a constant multiple of 
$$
|\bb|^{-1} \diam(\bar{K})\big(\|w\|_{L_2(K)}+\|u\|_{L_2(K)}\big).
$$
The terms on the second and third lines are bounded by  constant multiples of 
$$
|\bb|^{-2} \diam(\bar{K})^2\big(\|\partial_\bb u\|_{L_2(\bar{K})}+\|\partial_\bb w\|_{L_2(\bar{K})})\big).
$$

Proceeding to the difference of the last lines in \eqref{optimaltest}, respectively \eqref{near-opt},    we have
\begin{align*}
\|(x,{\bf y})\mapsto|\bb| \frac{w(x_{\!+}({\bf y}),{\bf y})-w(x_{\!-}({\bf y}),{\bf y})}{x_{\!+}({\bf y})-x_{\!-}({\bf y})}- \partial_{\bb} w(\bar{x}_{\!-}({\bf y}),{\bf y})\|_{L_2(K)}
\\ \lesssim |\bb|^{-1} \diam(\bar{K}) \|\partial_\bb^2 w\|_{L_2(\bar{K})},
\end{align*}
and
\begin{align*}
\|(x,{\bf y})\mapsto \frac{\int_{x_{\!-}({\bf y})}^{x_{\!+}({\bf y})} (\cc u+ \dd w) (q,{\bf y}) dq}{x_{\!+}({\bf y})-x_{\!-}({\bf y})}
-(\cc u(\bar{x}_{\!-}({\bf y}),{\bf y})+\dd w(\bar{x}_{\!-}({\bf y}),{\bf y}))\|_{L_2(K)}
\\
\lesssim |\bb|^{-1} \diam(\bar{K}) ( |\cc| \|\partial_\bb u\|_{L_2(\bar{K})}+ |\dd| \|\partial_\bb w\|_{L_2(\bar{K})}).
\end{align*}

Secondly, we find upper bound for the $L_2(K)$-norms for the different terms in $\partial_{\bb} (t_K-\breve{t}_K)$.
Since
\begin{align*}
\partial_\bb t_K(x,{\bf y})=&-
\Big[
u(x,\by)-u(x_{\!-}(\by),\by)+|\bb|^{-1} \int_{x_{\!-}({\bf y})}^x (\cc u + \dd w)(q,{\bf y}) d q 
\Big]
\\
&+ w(x_{\!+}({\bf y}),{\bf y})-u(x_{\!-}({\bf y}),{\bf y})+|\bb|^{-1} \int_{x_{\!-}({\bf y})}^{x_{\!+}({\bf y})} (\cc u + \dd w)(q,{\bf y}) dq,
\intertext{and}
\partial_\bb \breve{t}_K(x,{\bf y})= & \,\,w(\bar{x}_{\!-}({\bf y}),{\bf y})-u(\bar{x}_{\!-}({\bf y}),{\bf y}),
\end{align*}
we derive that
$$
\|(x,{\bf y})\mapsto u(x_{\!-}(\by),\by)-u(x,\by)\|_{L_2(K)} \lesssim |\bb|^{-1} \diam(K)  \|\partial_\bb u\|_{L_2(K)},
$$
\begin{align*}
\|(x,{\bf y})\mapsto 
|\bb|^{-1} \int_{x}^{x_{\!+}({\bf y})} (\cc u + \dd w)(q,{\bf y}) dq \|_{L_2(K)} \\
\lesssim |\bb|^{-1} \diam(K) (|\cc|\|u\|_{L_2(K)}+|\dd|\|w\|_{L_2(K)}),
\end{align*}
and
\begin{align*}
\|(x,{\bf y})\mapsto w(x_{\!+}(\by),\by)-w(\bar{x}_{\!-}(\by),\by)+u(\bar{x}_{\!-}(\by),\by)-u({x}_{\!-}(\by),\by)\|_{L_2(K)} \\
\lesssim |\bb|^{-1} \diam(\bar{K})  (\|\partial_\bb w\|_{L_2(\bar{K})}+\|\partial_\bb u\|_{L_2(\bar{K})}).
\end{align*}

The proof is completed by collecting all upper bounds,  by using that $\diam(\bar{K}) \lesssim \diam(K)$ (\eqref{108}), and that, for any polynomial $p$, $\|p\|_{L_2(\bar{K})} \lesssim \frac{|\bar{K}|}{|K|} \|p\|_{L_2({K})}$ with a constant depending  only on its degree.
\end{proof}

 As discussed earlier  below \eqref{wish}, inf-sup stability of a perturbed bilinear form $\breve{b}_h$ with respect to a given piecewise polynomial trial space and corresponding   test space based on \eqref{near-opt} will be partly established by direct calculations.  The next major step is given by the following lemma.

\begin{lemma} 
\label{lem6} 
Let $u|_K$ and $w|_K$ be polynomials of degree $m$ and assume that $\diam(K) \leq |\bb|$. For any $\eps >0$, one has
\begin{align}
\label{lower}
\|\breve{t}_K\|^2_{H(\bb;K)}
-\Big[{\textstyle \frac{1}{2+4|\cc|^{2}}} \|\partial_\bb w+(\cc+\dd)w\|_{L_2(K)}^2+
{\textstyle \frac{\eps}{2+ {8}\eps}} \|u\|_{L_2(K)}^2-
\eps\|w\|_{L_2(K)}^2\Big]\nonumber\\
 \gtrsim  {-}|\bb|^{-2} \diam(K)^2[\|u\|^2_{L_2(K)}+\|w\|^2_{H(\bb;K)}+\|\partial_\bb u\|^2_{L_2(K)}+\|\partial_\bb^2 w\|^2_{L_2(K)}
 ],
\end{align}
where the constant  depends only on upper bounds for $m$, $|\cc|$, $|\dd|$ and $\varrho_K$.
\end{lemma}

\begin{proof} 
By applying Young's inequality twice, in the form $\|\sigma\|^2 \geq (1-\eta) \|\tau\|^2+(1-\eta^{-1})\|\sigma-\tau\|^2$ for $\eta \in (0,1)$, here for $\eta=\frac{1}{2}$, we have
\begin{align*}
\|\breve{t}_K\|^2_{H(\bb;K)} =&
\|\breve{t}_K\|^2_{L_2(K)}+\|\partial_\bb \breve{t}_K\|^2_{L_2(K)}\\
\geq  &{\textstyle \frac{1}{2}}\big[\|\partial_\bb w+(\cc+\dd)w\|_{L_2(K)}^2+\|w-u\|_{L_2(K)}^2\big]\\
&-
\big[\|\breve{t}_K-(\partial_\bb w+(\cc+\dd)w)\|_{L_2(K)}^2
+\|\partial_\bb \breve{t}_K -(w-u)\|_{L_2(K))}^2\big].
\end{align*}
The same arguments that were used in the proof of Lemma~\ref{lem3} show that
\begin{align*}
\|\breve{t}_K&-(\partial_\bb w+(\cc+\dd)w)\|_{L_2(K)} \lesssim  \\
|\bb|^{-1} &\diam(K)\Big\{\|\partial_\bb^2 w\|_{L_2(K)}+|\cc|\|\partial_\bb u\|_{L_2(K)} +|\dd| \|\partial_\bb w\|_{L_2(K)}
\\
&+\|w\|_{L_2(K)}+\|u\|_{L_2(K)}+|\bb|^{-1}\diam(K)\big(\|\partial_\bb w\|_{L_2(K)}+\|\partial_\bb u\|_{L_2(K)}\big)\Big\},
\end{align*}
and
$$
\|\partial_\bb \breve{t}_K -(w-u)\|_{L_2(K))} \lesssim 
|\bb|^{-1} \diam(K)[\|w\|_{L_2(K)}+\|\partial_\bb u\|_{L_2(K)}].
$$

Recalling that $\cc$ is constant on $K$ and taking $\eta=1-\frac{1}{1+2|\cc|^{2}}$,   two applications of Young's inequality provide
\begin{align*}
&\|w-u\|_{L_2(K)}^2+\|\partial_{\bb} w+\cc u+\dd w\|_{L_2(K)}^2 \\
&\geq
\|w-u\|_{L_2(K)}^2+(1-\eta)\|\partial_{\bb} w+(\cc+\dd) w\|_{L_2(K)}^2+(1-{\textstyle \frac{1}{\eta}}) \|\cc(u-w)\|_{L_2(K)}^2\\
&={\textstyle \frac{1}{2}} \|w-u\|_{L_2(K)}^2+{\textstyle \frac{1}{1+2|\cc|^{2}}}\|\partial_{\bb} w+(\cc+\dd) w\|_{L_2(K)}^2\\
&\geq 
{\textstyle \frac{\eps}{1+4\eps}} \|u\|_{L_2(K)}^2-
2\eps\|w\|_{L_2(K)}^2+{\textstyle \frac{1}{1+2|\cc|^{2}}}\|\partial_{\bb} w+(\cc+\dd) w\|_{L_2(K)}^2,
\end{align*}
with which the proof is easily completed.
\end{proof}

\subsection{The main result} \label{the_main_result} \mbox{}

Let us fix
\begin{equation} 
\label{24}
{\bf b} \in W^1_\infty(\divv;\Omega),\, c \in W_\infty^1(\Omega) \text{ such that \eqref{14} is valid, and  } |{\bf b}|^{-1}\in L_\infty(\Omega),
\end{equation}
and with that, for any partition $\Omega_h$ of $\Omega$, the bilinear form $b_h$ given in \eqref{16}.
For any $K \in \Omega_h$, we set
$$
\bb_K:=|K|^{-1} \int_K {\bf b}\,d {\bf x}, \quad
\cc_K:=|K|^{-1} \int_K c -\divv {\bf b}\,d {\bf x},\quad
\dd_K:=|K|^{-1} \int_K \divv {\bf b}\,d {\bf x},
$$
and define ${\bf b}_h \in L_\infty(\Omega)^n$, $c_h \in L_\infty(\Omega)$, and $d_h \in L_\infty(\Omega)$ by
\begin{equation} \label{202}
{\bf b}_h|_K:=\bb_K, \quad c_h|_K:=\cc_K, \quad d_h|_K:=\dd_K \quad (K \in \Omega_{h}),
\end{equation}
with which we have defined the perturbed bilinear form $\breve{b}_h$ given in \eqref{barbhw0}--\eqref{200}.

Our subsequent analysis of the terms on the right hand side of \eqref{split}  along the strategy outlined in Section \ref{ssec:roadmap}  is guided by the following comments.
First, note that generally ${\bf b}_h \not\in W_\infty^0(\divv;\Omega)$,
  meaning that well-posedness of the corresponding variational form on the infinite dimensional level is not ensured.
 Indeed,
since for $\phi \in C_0^\infty(\Omega)$, $\int_\Omega \phi \divv {\bf b}_h \,d{\bf x}=\int_\Omega \phi \divv_h {\bf b}_h\,d{\bf x}+\int_{\partial\Omega_h} \llbracket {\bf b}_h \rrbracket \phi$, and, unless ${\bf b}_h$ is constant over $\Omega$, the right hand side cannot be bounded by a multiple of $\|\phi\|_{L_1(\Omega)}$, we have $\divv {\bf b}_h \not\in L_\infty(\Omega)$.
However, the perturbed form $\breve{b}_h$ is only applied to functions from finite dimensional spaces, which is also essential for treating the second summand in \eqref{split}.

In this latter regard, another problem is that ${\bf b}_h$ is an approximation to ${\bf b}$ that is only {\em first order accurate}.
In order to show that for a piecewise polynomial trial space, the second summand in \eqref{split} is sufficiently small relative to the first one, 
a central ingredient is to show that
for a piecewise polynomial $w_h$, $\frac{\int_\Omega ({\bf b}-{\bf b}_h)\cdot \nabla w_h}{\|w\|_{H({\bf b};\Omega_h)}}$ is sufficiently small. A combination of $\|{\bf b}-{\bf b}_h\|_{L_\infty(K)} \lesssim \diam(K)$, and the inverse inequality $|w|_{H^1(K)} \lesssim \diam(K)^{-1}\|w\|_{L_2(K)}$  shows only that this quotient is bounded.

We are going to solve this problem by considering trial spaces that are piecewise polynomial w.r.t. trial (macro-)partitions $\Omega_H$, 
such that the ratio of the local mesh sizes $h/H$ is less than some sufficiently small constant.
This will allow us also to take care of those `higher order' terms in Lemma~\ref{lem3}
which involve derivatives of $u$ and $w$. 

Specifically, let $\{\Omega_H:H \in \mathcal{I}\}$ be a family of partitions of a polyhedron $\Omega \subset \R^n$ into uniformly shape regular $n$-simplices, meaning that
\begin{equation} \label{25}
\varrho:=\sup_{H \in {\mathcal I}}\max_{K' \in \Omega_H} \varrho_{K'}<\infty.
\end{equation}
For any $H \in {\mathcal I}$, let $\Omega_h=\Omega_{h(H)}$ be a refinement of $\Omega_H$.
We set 
\begin{equation} \label{26}
\sigma:=\sup_{H \in \mathcal{I}} \max_{K' \in \Omega_H} \Big(\max_{\{K \in \Omega_{h}\colon K \subset K'\}} \frac{\diam(K)}{\diam(K')},\diam(K')\Big),
\end{equation}
which later will be assumed to be sufficiently small.
This means that we will assume that any partition $\Omega_H$ is sufficiently fine, and that the
 (minimal) {\em subgrid refinement factor} when going from any $\Omega_H$ to $\Omega_{h}$ is sufficiently large.
We consider only regular refinements $\Omega_h$ of $\Omega_H$, in the sense that
\begin{equation} \label{25b}
\bar{\varrho}:=\sup_{H \in {\mathcal I}}\max_{K \in \Omega_{h}} \varrho_K \lesssim \varrho,
\end{equation}
uniformly in $\sigma$.

Given $u,\,w \in \prod_{K \in \Omega_{h}} \cP_m(K)$, let $t=\cT(u,w),\,  {\breve{t}} \in H({\bf b}_h;\Omega_{h})$ be defined  for $K \in \Omega_{h}$
by
\begin{equation} \label{23}
t|_K:=\cT_{K,\bb_K,\cc_K,\dd_K}(u|_K,w|_{K}),\,\,
\breve{t}|_K:= {\breve{t}_K} \in \cP_{m+1}(K),
\end{equation}
 so that  $t|_K$ is the optimal local test function defined in \eqref{optimaltest} corresponding to the approximate, constant coefficients $\bb_K$, $\cc_K$, and $\dd_K$, and the replacement of the standard scalar product on $H(\bb;K)$ by $\langle\!\langle \,,\,\rangle\!\rangle_{K,\bb}$; and $\breve{t}|_K$ is its polynomial approximation defined in \eqref{near-opt}.

We can now formulate the main result of this paper.
\begin{theorem} \label{th1} 
Assume  the validity of \eqref{25}, \eqref{25b}, and \eqref{24}.
 Then there exists a $\sigma_0>0$ such that for $0<\sigma\leq \sigma_0$ (i.e., for sufficiently fine $\Omega_H$ and sufficiently large fixed subgrid refinement depth) 
 $$
(u,w) \in \U^H:=\prod_{K \in \Omega_H} \cP_m(K)\times  H_{0,\Gamma_{\!-}}({\bf b};\Omega) \cap  \prod_{K \in \Omega_H} \cP_{m}(K),
$$
and with $\breve{t}=\breve{t}(u,w) \in \prod_{K \in \Omega_h} \cP_{m+1}(K)$ as defined in \eqref{23}, it holds that
$$
b_h(u,w|_{\partial\Omega_h};\breve{t})
\gtrsim
\|(u,w)\|_\U \|\breve{t}\|_\V,
$$
where the constant depends only  on (upper bounds for)
$m$, $\varrho$, $\bar{\varrho}$, $\|{\bf b}\|_{W^1_\infty(\divv;\Omega)}$, $\||{\bf b}|^{-1}\|_{L_\infty(\Omega)}$, $\|c\|_{W^1_\infty(\Omega)}$, and 
$\|\cB^{-1}\|_{\cL(L_2(\Omega),H_{0,\Gamma_{\!-}}({\bf b};\Omega))}$.
\end{theorem}

The remainder of this section is devoted to the proof of this theorem.  
We begin with collecting some simple frequently needed technical preliminaries.

Obviously, we have
$$
\|c_h\|_{L_\infty(\Omega)} \leq \|c-\divv {\bf b}\|_{L_\infty(\Omega)},\quad\|d_h\|_{L_\infty(\Omega)} \leq \|\divv {\bf b}\|_{L_\infty(\Omega)}.
$$
Moreover, for any $n$-simplex $K \subset \Omega$, it holds that
\begin{equation} \label{20}
\begin{split}
\|\cc_K -(c-\divv {\bf b}))\|_{L_\infty(K)} &\lesssim   \diam(K)  |c-\divv {\bf b}|_{W_\infty^1(K)},\\
\|\dd_K -\divv {\bf b}\|_{L_\infty(K)} &\lesssim   \diam(K)  |\divv {\bf b}|_{W_\infty^1(K)},\\
\| |\bb_K-{\bf b}|\|_{L_\infty(K)} &\leq D   \diam(K)  |{\bf b}|_{W_\infty^1(K)^n},
\end{split}
\end{equation}
where, as the constant in the first two inequalities, $D>0$ is some constant depending only on $n$,
which we name for use in \eqref{102} below. 

In particular, we let
\begin{equation} \label{111}
\bar{\sigma}>0
\end{equation}
be such that for any $0 <\sigma\leq \bar{\sigma}$ and $H  \in {\mathcal I}$, $\Omega_{h}$ is sufficiently fine to ensure that
\begin{equation} \label{100}
 \diam(K)\, \||{\bf b}|^{-1}\|_{L_\infty(K)} \max\big(1,D |{\bf b}|_{W_\infty^1(K)^n}\big) \leq {\textstyle \frac{1}{2}} \quad (K \in \Omega_{h}).
\end{equation}
 Then for any $K \in \Omega_h$, we have
\begin{equation} \label{102}
\begin{split}
|\bb_K| &\geq \||{\bf b}|^{-1}\|_{L_\infty(K)}^{-1}-\||\bb_{K}-{\bf b}|\|_{L_\infty(K)}  \\
&\geq \||{\bf b}|^{-1}\|_{L_\infty(K)}^{-1} -D \diam(K) |{\bf b}|_{W_\infty^1(K)^n}\\
&\geq {\textstyle \frac{1}{2}}  \||{\bf b}|^{-1}\|_{L_\infty(K)}^{-1} \geq \max\big({\textstyle \frac{1}{2}}  \||{\bf b}|^{-1}\|_{L_\infty(\Omega)}^{-1},\diam(K)
\big),
\end{split}
\end{equation}
where we have used \eqref{100}.

Finally, for $H \in {\mathcal I}$, $K \in \Omega_H \cup \Omega_h$, and $k \geq \ell \in \N_0$, we will make repeated use of the {\em inverse inequality}
$$
|\cdot|_{H^k(K)} \lesssim \diam(K)^{-(k-\ell)}\|\cdot\|_{H^\ell(K)}\quad \text{on } \cP_m(K),
$$
where the constant depends only   on $m$, $\varrho$,  $\bar{\varrho}$, and $k$.\\

The main technical ingredients needed to prove Theorem \ref{th1} are collected in the following lemma.

\begin{lemma}
 \label{lem4}
Assume \eqref{25}, \eqref{25b}, and \eqref{24}. Then there exists a $0 < \sigma_0 \leq \bar{\sigma}$ (cf. \eqref{111}), such that for any
$\sigma \leq \sigma_0$, one has for all $(u,w) \in \prod_{K \in \Omega_H} \cP_m(K) \times   \prod_{K \in \Omega_H} \cP_{m}(K)   \cap H_{0,\Gamma_{\!-}}({\bf b};\Omega))$
\begin{align}
\label{mainest}
\begin{split}
& \|\breve{t}\|_{H({\bf b}_h;\Omega_{h})} \gtrsim \|(u,w)\|_\U, \quad \|t-\breve{t}\|_{H({\bf b}_h;\Omega_{h})} \lesssim \sigma \|\breve{t}\|_{H({\bf b}_h;\Omega_{h})},\\[1mm]
&\sum_{K \in \Omega_h} \diam(K)^2 \|\breve{t}|_K\|^2_{H^1(K)} \lesssim \sigma^2 \|\breve{t}\|^2_{H({\bf b}_h;\Omega_{h})},
\end{split}
\end{align}
where the constants depend 
only on (upper bounds for)
$m$, $\varrho$, $\bar{\varrho}$, $\|{\bf b}\|_{W^1_\infty(\divv;\Omega)}$, $\||{\bf b}|^{-1}\|_{L_\infty(\Omega)}$, $\|c\|_{W^1_\infty(\Omega)}$, and 
$\|B^{-1}\|_{\cL(L_2(\Omega),H_{0,\Gamma_{\!-}}({\bf b};\Omega))}$.
\end{lemma}
We defer the proof of this lemma to the end of this section and show first how it is used to complete the proof of Theorem~\ref{th1}
 following steps (i) and (ii) announced in Sect.~\ref{ssec:roadmap}.

\begin{proof}[Proof of Theorem~\ref{th1}] For the selection of ${\bf b}_h$, $c_h$ and $d_h$ from \eqref{202}, the perturbed bilinear form
on $(L_2(\Omega)\times H({\bf b};\Omega)) \times H({\bf b}_h;\Omega_{h})$, first mentioned in \eqref{barbhw0}--\eqref{200}, reads as
\begin{align*}
\breve{b}_h(u,w;v)&:=\int_\Omega (c_h v-{\bf b}_h \cdot \nabla_h v)u+d_h v w \,d{\bf x}+\int_{\partial \Omega_{h}} \llbracket v {\bf b}_h\rrbracket w\,d{\bf s}\\
&
=\sum_{K \in \Omega_{h}} \int_K \cc_K v u+(w-u) \bb_K \cdot \nabla v+v \bb_K \cdot \nabla  w+\dd_K v w \,d {\bf x}.
\end{align*}
 Recall from \eqref{18} that the optimal test function $t$, defined in   \eqref{23}, was constructed such that
$$
\sum_{K \in \Omega_h}
\langle\!\langle t|_K,v|_K \rangle\!\rangle_{K,\bb_K}=\breve{b}_h(u,w;v) \quad (v \in H({\bf b}_h;\Omega_{h})).
$$
Therefore, since for  $\sigma \leq \bar{\sigma}$, $\diam(K) \leq |\bb_K|$ by \eqref{102}, upon taking $\sigma_0 \leq \bar{\sigma}$, Proposition~\ref{prop2} applies and Remark \ref{rem:4.3}
ensures 
that
\begin{equation} \label{30}
\|t\|^2_{H({\bf b}_h;\Omega_h)} \eqsim \sum_{K\in \Omega_h} \langle\!\langle t|_K,t|_K \rangle\!\rangle_{K,\bb_K}=\breve{b}_h(u,w;t).
\end{equation}

For $(u,w) \in \U^H$, applying the inverse inequality in combination with  \eqref{20}, shows that  $\|(\bb_K-{\bf b}|_K) \cdot \nabla \breve{t}|_K\|_{L_2(K)} \lesssim \|\breve{t}|_K\|_{L_2(K)}$ so that
\begin{equation} \label{27}
\|\breve{t}\|_{H({\bf b}_h;\Omega_h)} \eqsim \|\breve{t}\|_{\V}.
\end{equation}

{For $\sigma_0>0$ sufficiently small, the second inequality in Lemma~\ref{lem4} gives $\|\breve t\|_{H(\bbf_h;\Omega_h)} \eqsim \|t\|_{H(\bbf_h;\Omega_h)}$.
We infer that
\begin{align*}
\breve{b}_h(u,w;t) \eqsim \|\breve{t}\|^2_{H({\bf b}_h;\Omega_h)} \eqsim \|\breve{t}\|_{H({\bf b}_h;\Omega_h)}\|\breve{t}\|_{\V}
\gtrsim \|(u,w)\|_\U \|\breve{t}\|_{\V},
\end{align*}
by  the first inequality in Lemma~\ref{lem4}.}

Since $\breve{b}_h$ is bounded on $\U \times H({\bf b}_h;\Omega_{h})$, uniformly in $h$, we have
\begin{equation} \label{29}
 |\breve{b}_h(u,w;\breve{t})-\breve{b}_h(u,w;t)|  \lesssim\|(u,w)\|_\U \|t-\breve{t}\|_{H({\bf b}_h;\Omega_h)}  \lesssim \sigma \|(u,w)\|_\U \|\breve{t}\|_{\V},
\end{equation}
where we have again used the second inequality in Lemma~\ref{lem4} and \eqref{27}. We conclude that for $\sigma\leq \sigma_0$ sufficiently small, 
$$
\breve{b}_h(u,w;\breve{t}) 
\gtrsim \|(u,w)\|_\U \|\breve{t}\|_{\V},
$$
which is step (i) from Sect.~\ref{ssec:roadmap}.

As for step (ii), we have for $(u,w) \in \U$ 
 $$
b_h(u,w|_{\partial\Omega_h};v):=\sum_{K \in \Omega_{h}} \int_K (c-\divv {\bf b}) v u+(w-u) {\bf b} \cdot \nabla v+v {\bf b} \cdot \nabla  w+ v w \divv {\bf b}\,d {\bf x}.
$$
Applying \eqref{20} and subsequently  the third inequality of \eqref{mainest} in Lemma~\ref{lem4},
we obtain  for $(u,w) \in \U^H$ 
\begin{equation} \label{28}
\begin{split}
& |b_h(u,w|_{\partial\Omega_h};\breve{t})-\breve{b}_h(u,w;\breve{t})|\\
& \lesssim \sum_{K \in \Omega_{h}} \diam(K) 
\big[
\|(u,w)\|_\U \|\breve{t}\|_{H^1(K)}
+\|\breve{t}\|_{L_2(K)} \|w\|_{H^1(K)}
\big]\\
& \lesssim \|(u,w)\|_\U \sqrt{\sum_{K \in \Omega_{h}} \diam(K)^2\|\breve{t}\|_{H^1(K)}^2}
\\&\hspace*{10em}+\|\breve{t}\|_{L_2(\Omega)} \sigma \sqrt{\sum_{K' \in \Omega_H} \diam(K')^2 \|w\|_{H^1(K')}^2}\\
& \lesssim \sigma \|(u,w)\|_\U \|\breve{t}\|_{H({\bf b}_h;\Omega_h)}
\lesssim \sigma \|(u,w)\|_\U\|\breve{t}\|_{\V}
\end{split}
\end{equation}
where we have applied
 the inverse inequality to $w|_{K'}$ for $K' \in \Omega_H$, and, finally \eqref{27}.
Estimate \eqref{28} is step (ii) from Sect.~\ref{ssec:roadmap} which, together with step (i)  completes the proof of Theorem \ref{th1}.
\end{proof}

\noindent
{\it Proof of Lemma \ref{lem4}:}
To show the first inequality in \eqref{mainest}, we will sum over $K \in \Omega_h$ the inequality \eqref{lower}   in Lemma~\ref{lem6}.
We start with showing below in \eqref{101} that the resulting right-hand side can be made small enough.
To exploit that $u$ and $w$ are piecewise polynomial w.r.t. the `coarse grid' $\Omega_H$, we collect all $K \in \Omega_h$ that are contained in one $K'\ \in \Omega_H$. 

To arrive at \eqref{101} we need, in particular, to get rid of the derivatives of $u$ and to switch from $\|w\|_{H(\bb;\Omega)}$ to $\|w\|_{H(\bbf;\Omega)}$.
To this end, an easy consequence of the third estimate in \eqref{20}  is $\||\bb_K|\|_{L_\infty(K)} \lesssim \|{\bf b}\|_{W^1_\infty(K)^n}$ for $K \in \Omega_h$.
Together with an application of the inverse inequality on $K' \in \Omega_H$, this shows that
\be
\label{first}
\sum_{\{K \in \Omega_h\colon K \subset K'\} } 
\|\partial_{\bb_K} u\|^2_{L_2(K)} \lesssim |u|_{H^1(K')}^2 \lesssim \diam(K')^{-2} \|u\|_{L_2(K')}^2,
\ee
with a constant depending on $m$, $\varrho$, and $\|{\bf b}\|_{W^1_\infty(\Omega)^n}$. Next, combining again the third inequality in \eqref{20}
with an inverse estimate on $K'' \in \{K,K'\}$ yields
\be
\label{wb}
\begin{split}
\|\partial_{\bb_{K''}}w\|^2_{L_2({K''})} & \le  2\big\{ \|\partial_{{\bb}}w\|^2_{L_2(K'')}   + \|(\bbf- \bb_{K''})\cdot\nabla w\|^2_{L_2(K'')}  \big\}
\\ & \lesssim  \| w\|_{H(\bbf;K'')}^2,
\end{split}
\ee
with a constant depending on $\rho$ or $\bar\rho$, and on $m,D,\|\bbf\|_{W^1_\infty({\Omega})^n}$.

The terms $\|\partial_{\bb_K}^2w\|_{L_2(K)}^2$ require a little more care  than $\|\partial_{\bb_K} u\|_{L_2(K)}^2$ since unlike $u$, $\partial_{\bb_K} w$ is generally not piecewise polynomial w.r.t. $\Omega_H$. 
 
Therefore, we first use that for $\Omega_h \ni K \subset K' \in \Omega_H$,
$$
\|\partial_{\bb_K}^2w\|^2_{L_2(K)} \le 2\big\{\|\partial_{\bb_K} (\partial_{\bb_{K}}-\partial_{\bb_{K'}}) w\|^2_{L_2(K)}+\|\partial_{\bb_K} \partial_{\bb_{K'}}w\|^2_{L_2(K)}\big\}
$$
For the second term on the right an application of \eqref{first} with $u$ reading as $\partial_{\bb_{K'}}w$ shows that 
\begin{align}
\label{second}
\sum_{\{K \in \Omega_h\colon K \subset K'\} } 
\|\partial_{\bb_K} \partial_{\bb_{K'}}w\|^2_{L_2(K)} & \lesssim \diam(K')^{-2} \|\partial_{\bb_{K'}}w\|_{L_2(K')}^2\nonumber   \\
& 
 \lesssim \diam(K')^{-2}\|w\|_{H({\bf b};K')}^2,
\end{align}
where we have used \eqref{wb} for $K''=K'$.
For the first term on the right we derive that
\begin{align}
\label{third}
&\|\partial_{\bb_K} (\partial_{\bb_{K}}-\partial_{\bb_{K'}}) w\|^2_{L_2(K)} =
\|(\partial_{\bb_{K}}-\partial_{\bb_{K'}})\partial_{\bb_K} w\|^2_{L_2(K)}\nonumber \\
& \lesssim \diam(K')^2 |\partial_{\bb_K} w|^2_{H^1(K)} 
\lesssim {  \frac{\diam(K')^2}{\diam(K)^2}} \|\partial_{\bb_K}w\|^2_{L_2(K)}\nonumber\\
&  \lesssim { \frac{\diam(K')^2}{\diam(K)^2}} \|w\|^2_{H({\bf b};K)},
\end{align}
where both \eqref{second} and \eqref{third} depend on $m$, $\varrho$, $\bar{\varrho}$, and $\|{\bf b}\|_{W^1_\infty(\Omega)^n}$.

By combining these {four} estimates \eqref{first}, {\eqref{wb} for $K''=K$}, \eqref{second}, \eqref{third}, and using $|\bb_K|^{-1} \leq 2\||{\bf b}|^{-1}\|_{L_\infty(\Omega)}$ (\eqref{102}), $\diam(K) \leq \sigma \diam(K')$, and $\diam(K') \leq \sigma$, we infer that
\begin{align*}
\sum_{\{K \in \Omega_h\colon K \subset K'\} } &\hspace*{-1em}
{\textstyle \frac{\diam(K)^2}{|\bb_K|^2}}\Big[\|u\|^2_{L_2(K)}+\|w\|^2_{H(\bb_K;K)}+\|\partial_{\bb_K} u\|^2_{L_2(K)}+\|\partial_{\bb_K}^2 w\|^2_{L_2(K)} \Big]\\
&\lesssim \sigma^2 \Big[\|u\|^2_{L_2(K')}+\|w\|^2_{H({\bf b};K')}\Big],
\end{align*}
and so
\begin{equation} \label{101}
\begin{split}
\sum_{K \in \Omega_h} &
{\textstyle \frac{\diam(K)^2}{|\bb_K|^2}}\Big[\|u\|^2_{L_2(K)}+\|w\|^2_{H(\bb_K;K)}+\|\partial_{\bb_K} u\|^2_{L_2(K)}+\|\partial_{\bb_K}^2 w\|^2_{L_2(K)} 
\Big]\\
&\lesssim \sigma^2 \|(u,w)\|_\U^2,
\end{split}
\end{equation}
where the constant depends on $m$, $\varrho$, $\bar{\varrho}$, $\|{\bf b}\|_{W^1_\infty(\Omega)^n}$, and $\||{\bf b}|^{-1}\|_{L_\infty(\Omega)}$.

To treat next the terms on the left hand side of \eqref{lower}
analogous arguments, preceded by applications of the triangle inequality, show that
\begin{align*}
&\Big| \|{\bf b}_h \cdot \nabla_h w+(c_h+d_h)w\|_{L_2(\Omega)}-\|{\bf b}\cdot \nabla w + cw\|_{L_2(\Omega)}
\Big|^2 \\
&\leq \sum_{K' \in \Omega_H}
\sum_{\{K \in \Omega_h\colon K \subset K'\} }\|(\bb_K-{\bf b})\cdot \nabla w+(\cc_K+\dd_K-c)w\|_{L_2(K)}^2\\
&\lesssim \sum_{K' \in \Omega_H} \sigma^2 \|w\|^2_{L_2(K')}=\sigma^2 \|w\|^2_{L_2(\Omega)},
\end{align*}
which, upon using 
 $\big|\|f\|^2-\|g\|^2\big|\leq \big|\|f\|-\|g\|\big|\big(2\|g\|+\big|\|f\|-\|g\|\big|\big)$, 
yields
\begin{equation} \label{105}
\begin{split}
&\Big| \|{\bf b}_h \cdot \nabla_h w+(c_h+d_h)w\|^2_{L_2(\Omega)}-\|{\bf b}\cdot \nabla w + cw\|^2_{L_2(\Omega)}
\Big|\\
& \lesssim \sigma \|w\|_{L_2(\Omega)}\big[(\|{\bf b}\cdot \nabla w + cw\|_{L_2(\Omega)}+\sigma \|w\|_{L_2(\Omega)}\big] 
 \lesssim \sigma \|w\|_{H({\bf b};\Omega)}^2
\end{split}
\end{equation}
dependent on $m$, $\varrho$, $\bar{\varrho}$, $\|{\bf b}\|_{W^1_\infty(\divv;\Omega)}$, and $\|c\|_{W^1_\infty(\Omega)}$.

Now by summing the inequality \eqref{lower} in Lemma~\ref{lem6} over $K \in \Omega_h$, substituting the estimates \eqref{101} and \eqref{105}, and using that
$$
\|w\|_{H({\bf b};\Omega)}\leq \|\cB^{-1}\|_{\cL(L_2(\Omega),H_{0,\Gamma_{\!-}}({\bf b};\Omega))}  \|{\bf b} \cdot \nabla w+c w\|_{L_2(\Omega)},
$$
for any $\eps>0$ we arrive at 

\begin{align*}
&\|\breve{t}\|_{H({\bf b}_h;\Omega_{h})}^2 - \Big[{\textstyle \frac{\|\cB^{-1}\|_{\cL(L_2(\Omega),H_{0,\Gamma_{\!-}}
({\bf b};\Omega))}^{-2}}{2+4
{\|c -\divv \bbf\|_{L_\infty(\Omega)}^2}
}} \|w\|_{H({\bf b};\Omega)}^2+{\textstyle \frac{\eps}{2+8 \eps}} \|u\|_{L_2(\Omega)}^2-\eps\|w\|_{L_2(\Omega)}^2\Big] \\
&\gtrsim
-\sigma^2 \|u\|_{L_2(\Omega)}^2-\sigma\|w\|_{H({\bf b};\Omega)}^2,
\end{align*}
with a constant depending on $m$, $\varrho$, $\|{\bf b}\|_{W^1_\infty(\divv;\Omega)}$, $\||{\bf b}|^{-1}\|_{L_\infty(\Omega)}$, and $\|c\|_{W^1_\infty(\Omega)}$.
By selecting $\eps$ and, subsequently, $\sigma_0$ small enough, the proof of the first estimate in \eqref{mainest} is completed.
\medskip

Lemma~\ref{lem3} in combination with \eqref{101} shows that
$$
\|t-\breve{t}\|_{H({\bf b}_h;\Omega_{h})} \lesssim \sigma \|(u,w)\|_\U.
$$
Now the second estimate follows from the first.
\medskip

To prove the last estimate, we split $\breve{t}=\breve{t}_1+\breve{t}_2+\breve{t}_3$ (see \eqref{near-opt}), where, for $K' \in \Omega_H$, $K \in \Omega_h$ with $K \subset K'$,
\begin{align*}
\breve{t}_1|_K(x,{\bf y})& :=|\bb_K|^{-1} \Big(w(\bar{x}_{\!-}({\bf y}),{\bf y})-u(\bar{x}_{\!-}({\bf y}),{\bf y}) \Big)\Big(x-\bar{x}_{\!-}({\bf y})\Big),\\
\breve{t}_2|_K(x,{\bf y})& :=\partial_{\bb_{K'}} w(\bar{x}_{\!-}({\bf y}),{\bf y}) + \cc_K u(\bar{x}_{\!-}({\bf y}),{\bf y})+\dd_K w(\bar{x}_{\!-}({\bf y}),{\bf y})\\
\breve{t}_3|_K(x,{\bf y})& :=(\bb_K-\bb_{K'})\cdot \nabla w(\bar{x}_{\!-}({\bf y}),{\bf y}).
\end{align*}

Since $\breve{t}_1|_K \in \cP_{m+1}(K)$ vanishes on $\partial \bar{K}_{\!-}$, the inverse inequality, Proposition~\ref{prop2} and \eqref{108} show that
\begin{equation} \label{103}
\begin{split}
\|\breve{t}_1|_K\|_{H^1(K)} & \lesssim  \diam(K)^{-1} \|\breve{t}_1|_K\|_{L_2(K)} 
\leq \diam(K)^{-1}  \|\breve{t}_1|_K\|_{L_2(\bar{K})} \\
& \lesssim {\textstyle \frac{\diam(\bar{K})}{\diam(K)}}
 \|\partial_{\bb_K} \breve{t}_1|_K\|_{L_2(\bar{K})} \lesssim  \|\partial_{\bb_K}\breve{t}|_K\|_{L_2(K)} \leq  \|\breve{t}|_K\|_{H(\bb_K;K)},
\end{split}
\end{equation}
with a constant depending on $\bar{\varrho}$.

To treat $\breve{t}_2$, let $K \in \Omega_h$ and $p \in \cP_m(K)$. Recalling from \eqref{104} that  $|\bar{x}_{\!-}|_{W^1_\infty(K)} \lesssim 1$, we have
\begin{align*}
\|{\bf x} \mapsto p(\bar{x}_{\!-}({\bf y}),{\bf y})\|_{H^1(K)} \lesssim |K|^{\frac{1}{2}} \|{\bf x} \mapsto p(\bar{x}_{\!-}({\bf y}),{\bf y})\|_{W^1_\infty(K)} \lesssim |K|^{\frac{1}{2}} \|p\|_{W^1_\infty(K)},
\end{align*}
also with a constant depending on  $\bar{\varrho}$.
Now consider a $p \in \cP_m(K')$ for a $K' \in \Omega_H$. Then the combination of the previous result and the inverse inequality on $K'$ show that
\begin{align*}
\sum_{\{K \in \Omega_h\colon K \subset K'\}} \diam(K)^2 \|{\bf x} \mapsto p(\bar{x}_{\!-}({\bf y}),{\bf y})\|_{H^1(K)}^2
\lesssim \sigma^2 \diam(K')^2 |K'| \|\|p\|^2_{W^1_\infty(K')}\\
\lesssim \sigma^2  |K'|  \|\|p\|^2_{L_\infty(K')}
\lesssim \sigma^2 \|p\|_{L_2(K')}^2,
\end{align*}
dependent on $\varrho$, $\bar{\varrho}$, and $m$.
By applying this to $\breve{t}_2$, we obtain
\begin{equation} \label{106}
\begin{split}
\sum_{K \in \Omega_h} \diam(K)^2 \|\breve{t}_2|_K\|_{H^1(K)}^2
&\lesssim \sigma^2 \Big[\|u\|_{L_2(\Omega)}^2+ \|w\|_{L_2(\Omega)}^2+\sum_{K' \in \Omega_H} \|\partial_{\bb_{K'}} w\|_{L_2(K')}^2\Big]\\
&\lesssim \sigma^2 \|(u,w)\|_\U^2 \lesssim \sigma^2 \|\breve{t}\|_{H({\bf b}_h;\Omega_h)}^2,
\end{split}
\end{equation}
whith a constant depending on $\varrho$, $\bar{\varrho}$, $m$, $\|c\|_{L_\infty(\Omega)}$, and $\|{\bf b}\|_{W^1_\infty(\divv;\Omega)}$, 
where we used \eqref{wb} in the second but last step
as well as the first inequality in \eqref{mainest} in the last step. 

For $\Omega_h \ni K \subset K' \in \Omega_H$, using \eqref{104} and $\|\bb_K-\bb_{K'}\|_{L_\infty(K)} \lesssim \diam(K') |{\bf b}|_{W^1_\infty(K')^n}$,
we estimate
\begin{align*}
&\sum_{\{K \in \Omega_h\colon K \subset K'\}} \diam(K)^2 \|\breve{t}_3|_K\|_{H^1(K)}^2\leq 
\sum_{\{K \in \Omega_h\colon K \subset K'\}} \diam(K)^2 |K| \|\breve{t}_3|_K\|_{W_\infty^1(K)}^2\\ 
& \lesssim \sum_{\{K \in \Omega_h\colon K \subset K'\}} \diam(K)^2 |K| \diam(K')^2 \|w|_{K'}\|_{W_\infty^2(K')}^2\\ 
&\lesssim \sum_{\{K \in \Omega_h\colon K \subset K'\}} \diam(K)^2 |K| \diam(K')^{-2} \|w|_{K'}\|_{L_\infty(K')}^2\\
&\leq \sigma^2|K'| \|w|_{K'}\|_{L_\infty(K')}^2 \lesssim 
\sigma^2 \|w|_{K'}\|_{L_2(K')}^2,
\end{align*}
with a constant depending on $\varrho$, $\bar{\varrho}$, $m$ and $|{\bf b}|_{W^1_\infty(K')^n}$.  Thus, we conclude that 
\begin{equation} \label{107}
\sum_{K \in \Omega_h} \diam(K)^2 \|\breve{t}_3|_K\|_{H^1(K)}^2 \lesssim \sigma^2 \|w\|_{L_2(\Omega)}^2 \lesssim \sigma^2 \|\breve{t}\|_{H({\bf b}_h;\Omega_h)}^2
\end{equation}
using again the first inequality in \eqref{mainest} of this lemma. Combining \eqref{103}, \eqref{106}, and \eqref{107}, completes the proof of the last claim of this lemma. \hfill $\Box$

\section{Some numerical results}
On $\Omega=(0,1)^2$, and for ${\bf b} \in W_\infty^1(\divv;\Omega)$ with $|{\bf b}|^{-1} \in L_\infty(\Omega)$ and ${\bf u} \mapsto {\bf b}\cdot \nabla {\bf u} \in \Lis(H_{0,\Gamma_{\!-}}({\bf b};\Omega),L_2(\Omega))$, we consider the transport problem
$$
\left\{
\begin{array}{r@{}c@{}ll}
{\bf b}\cdot \nabla u &\,\,=\,\,& f &\text{ on } \Omega,\\
u&\,\,=\,\, & 0 &\text{ on } \Gamma_{\!-}.
\end{array}
\right.
$$
We let $\Omega_H$ be a partition of $\Omega$ into uniformly shape regular triangles, and let $\Omega_h$ be the refinement of $\Omega_H$ by applying $\ell$ recursive red-refinements to each $K \in \Omega$ where $\ell \in \N_0$ is a fixed number.
We let
\begin{align*}
\U^H &=\prod_{K \in \Omega_H} \cP_{m-1}(K) \,\,\times \,\,
\Big\{w|_{\partial\Omega}\colon w \in C(\Omega) \cap \prod_{K \in \Omega_H} \cP_{m}(K),\,w=0 \text{ on }\Gamma_{\!-}\Big\},\\
\V^h &=\prod_{K \in \Omega_h} \cP_{m+1}(K).
\end{align*}
We solve $(u^H,\theta^H) \in \U^H$ from
\begin{equation} \label{90}
\begin{split}
b_h(u^H,\theta^H;v^h) := &-\int_\Omega ({\bf b}\cdot\nabla_h v^h +v^h \divv {\bf b})u^H \,d{\bf x}+\int_{\partial\Omega_h}  \llbracket v^h {\bf b} \rrbracket \theta^H \,d{\bf s} \\
=& \, f(v^h)
\qquad(v^h \in  {\mathcal T}^h(\U^H)),
\end{split}
\end{equation}
with ${\mathcal T}^h \in \cL(L_2(\Omega) \times H_{0,\Gamma_{\!-}}({\bf b};\partial\Omega_h),\V^h)$ being defined by
\begin{equation} \label{93}
\langle {\mathcal T}^h (u,\theta),v^h\rangle_{H({\bf b};\Omega_h)} =b_h(u,\theta;v^h) \quad (v^h \in \V^h).
\end{equation}
Note that  for $(u,\theta)$ running over an obvious localized basis for $\U^H$, finding each of the ${\mathcal T}^h (u,\theta)$ amounts to solving a fixed finite dimensional problem on a few mesh cells.
Having determined such a basis for ${\mathcal T}^h(\U^H)$, the solution of \eqref{90} can be found by solving the sparse, symmetric positive definite system
$$
\langle {\mathcal T}^h (u^H,\theta^H),{\mathcal T}^h (\tilde{u}^H,\tilde{\theta}^H)\rangle_{H({\bf b};\Omega_h)} = f({\mathcal T}^h (\tilde{u}^H,\tilde{\theta}^H)) \quad ((\tilde{u}^H,\tilde{\theta}^H) \in U^H).
$$

As shown in Theorem~\ref{th1}, by taking a sufficiently large, but fixed $\ell$,
\begin{equation}
\begin{split} \label{91}
\|u-u^H\|_{L_2(\Omega)}+&\|\theta-\theta^H\|_{H_{0,\Gamma_{\!-}}({\bf b};\partial\Omega_h)}\\
&  \lesssim \inf_{(\bar{u}^H,\bar{\theta}^H) \in \U^H}\big\{
\|u-\bar{u}^H\|_{L_2(\Omega)}+\|\theta-\bar{\theta}^H\|_{H_{0,\Gamma_{\!-}}({\bf b};\partial\Omega_h)}\big\}.
\end{split}
\end{equation}

In all our experiments, we only measure $\|u-u^H\|_{L_2(\Omega)}$, being the quantity of our main interest.
We report on cases where $m=1$, so piecewise constant approximations for $u$, and piecewise linear approximations for $\theta$.
It appears that in all these cases it is   sufficient to take $\ell=0$, i.e., $\Omega_h=\Omega_H$.
Increasing $\ell$ leaves the numerical solutions essentially unchanged. This holds for true for $m=1$, as well as in experiments that we performed where $m>1$.

In our first experiment, we take constant ${\bf b}=(b_1,b_2)^\top \in \R_{>0}\times \R_{\geq 0}$, $\Omega_H$ being a  uniform partition of $\Omega$  into isosceles right angled triangles with legs of length $H\in 2^{-\N_0}$ and hypothenuses parallel
to the vector $(1,1)$, and $\Omega_h=\Omega_H$ so $\ell=0$. 
We take $f({\bf x})=1-x_1$ so that the exact solution, given by
$$
u({\bf x})=\left\{
\begin{array}{ll}
 \frac{x_1}{b_1}-\frac{x_1^2}{2 b_1}, &  -b_2 x_1+b_1 x_2\ \geq  \ 0,  \\
  \frac{x_2}{b_2}-\frac{x_2 (2 b_2 x_1-b_1 x_2)}{2 b_2^2},&  -b_2 x_1+b_1 x_2\ <  \ 0, 
 \end{array} \right.
$$
is continuous, piecewise quadratic, whose normal derivative over the line $ {\bf x} \cdot {\bf b}^\perp=0$   has a jump.
The numerical results for various ${\bf b}$, illustrated in Figure~\ref{error1} for ${\bf b}=(1,1)^\top$ and ${\bf b}=(1,1/16)^\top$,
\begin{figure}[h]
\includegraphics[width=6cm]{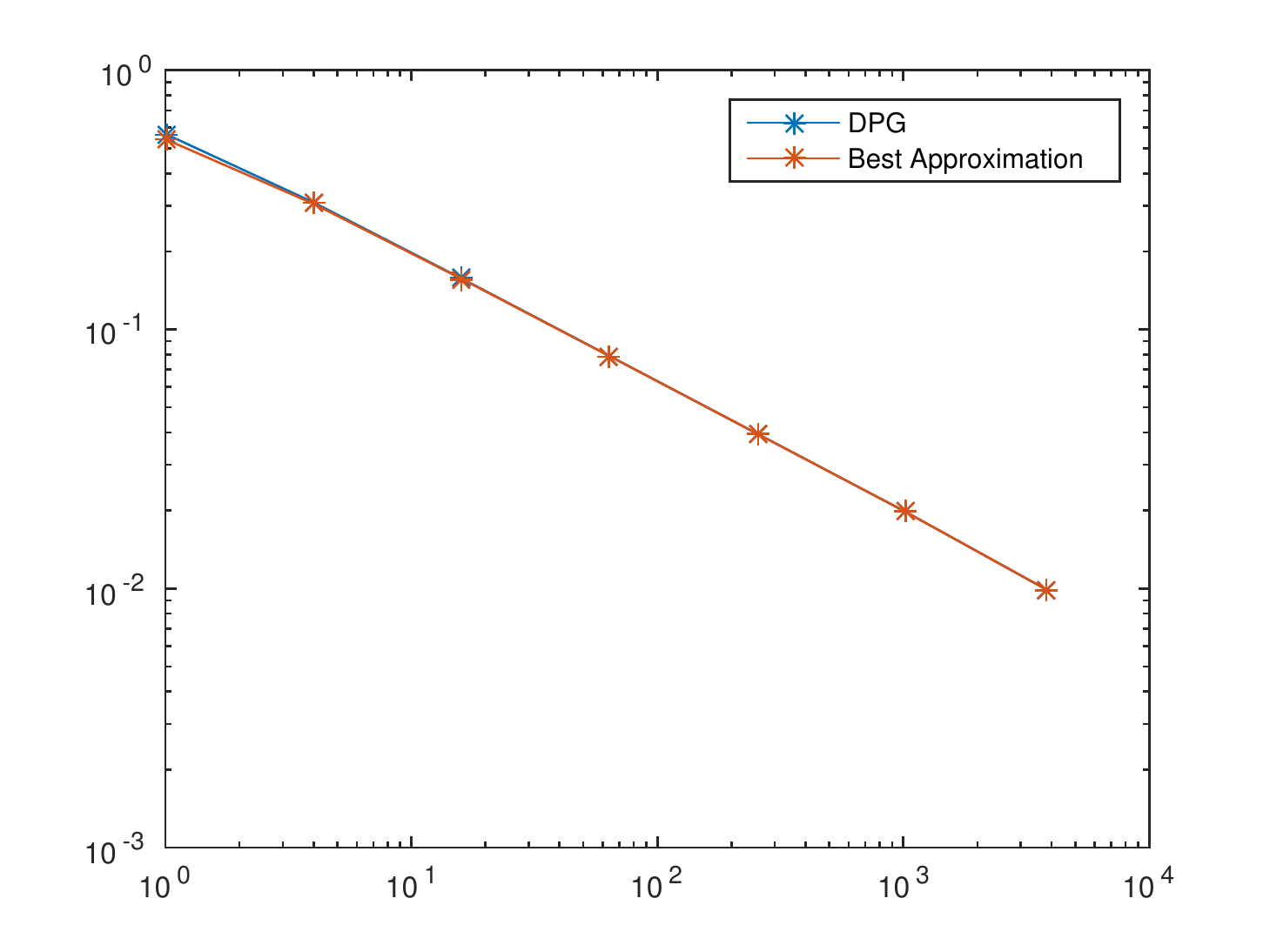}
\includegraphics[width=6cm]{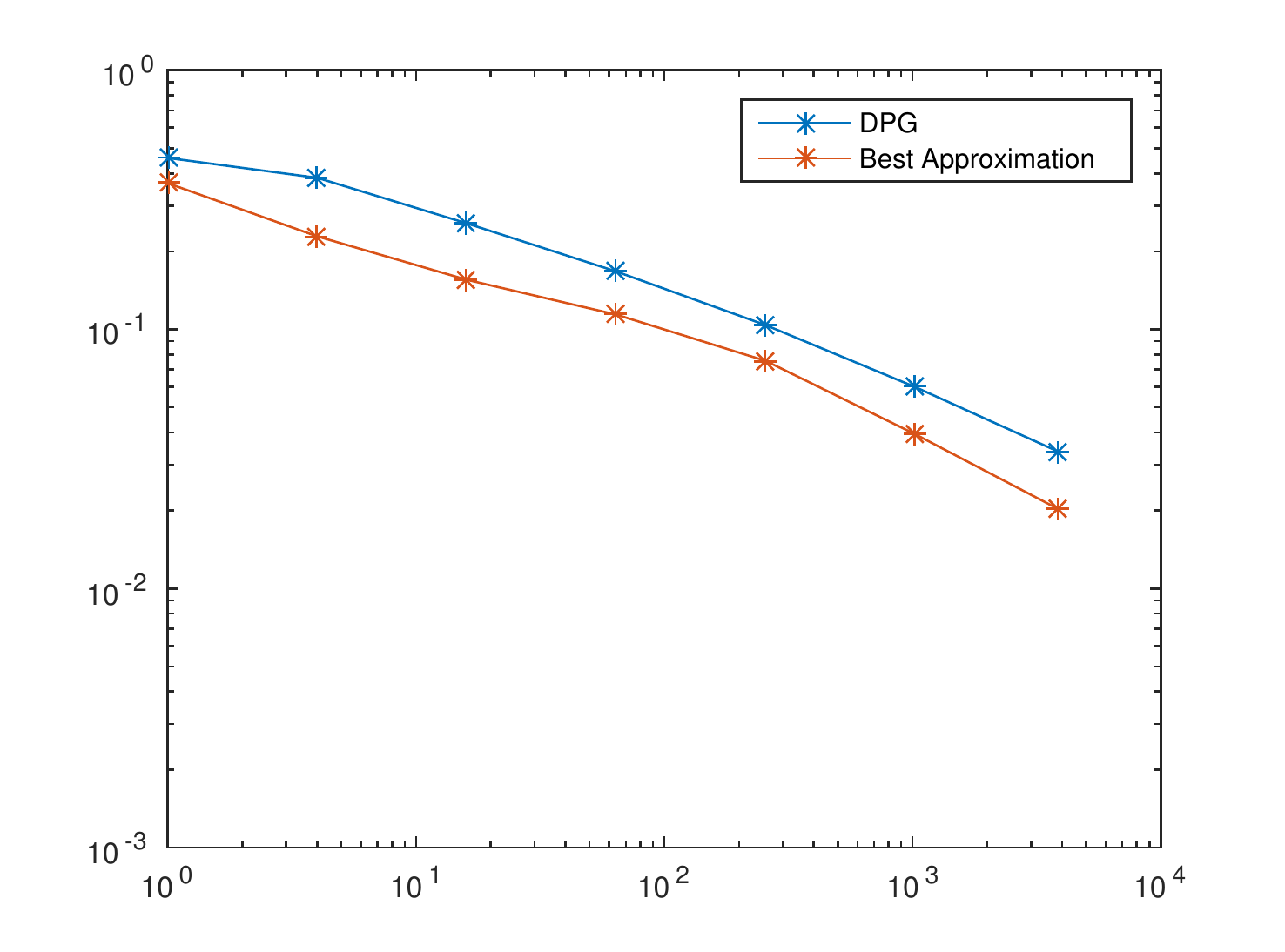}
\caption{$L_2(\Omega)$-error in $u_H$ and that in the best approximation versus $1/h^2$, for $f({\bf x})=1-x_1$, ${\bf b}=(1,1)^\top$ (left) and ${\bf b}=(1,\frac{1}{16})^\top$ (right).}
\label{error1}
\end{figure}
show that $\|u-u^H\|_{L_2(\Omega)}$ is close to the error of best approximation from the space of piecewise constants.

In our second experiment, we change $f$ into
\begin{equation} \label{92}
f({\bf x})=\left\{
\begin{array}{ll}
 1-x_1, &  -b_2 x_1+b_1 x_2\ \geq  \ \frac{1}{4},  \\
 0,&  -b_2 x_1+b_1 x_2\ <  \ \frac{1}{4}, 
 \end{array} \right.
\end{equation}
so that the solution, given by
$$
u(x_1,x_2)=\left\{
\begin{array}{lr}
 \frac{x_1}{b_1}-\frac{x_1^2}{2 b_1}, & -b_2 x_1+b_1 x_2\ \geq  \ \frac{1}{4},  \\
 0,&  -b_2 x_1+b_1 x_2\ <  \ \frac{1}{4},
 \end{array}
 \right.
 $$
 is piecewise quadratic with a discontinuity over the line  $ {\bf x} \cdot {\bf b}^\perp=\frac{1}{4}$.
 
 When $h=2^{-k}$ for $k \geq 2$ and ${\bf b} \in \{(1,0)^\top,(1,1)^\top\}$, then this discontinuity is over a grid line, and the right-hand side of \eqref{91} will be strongly dominated by the approximation error in $\theta$, because the approximation error in $u$ benefits from the discontinuous approximation. In this, rather special situation, the error $\|u-u_H\|_{L_2(\Omega)}$ might therefore be much larger than the error of best approximation in $u$. Unfortunately, this is indeed what happens as illustrated in Figure~\ref{error2}.
 \begin{figure}[h]
\includegraphics[width=6cm]{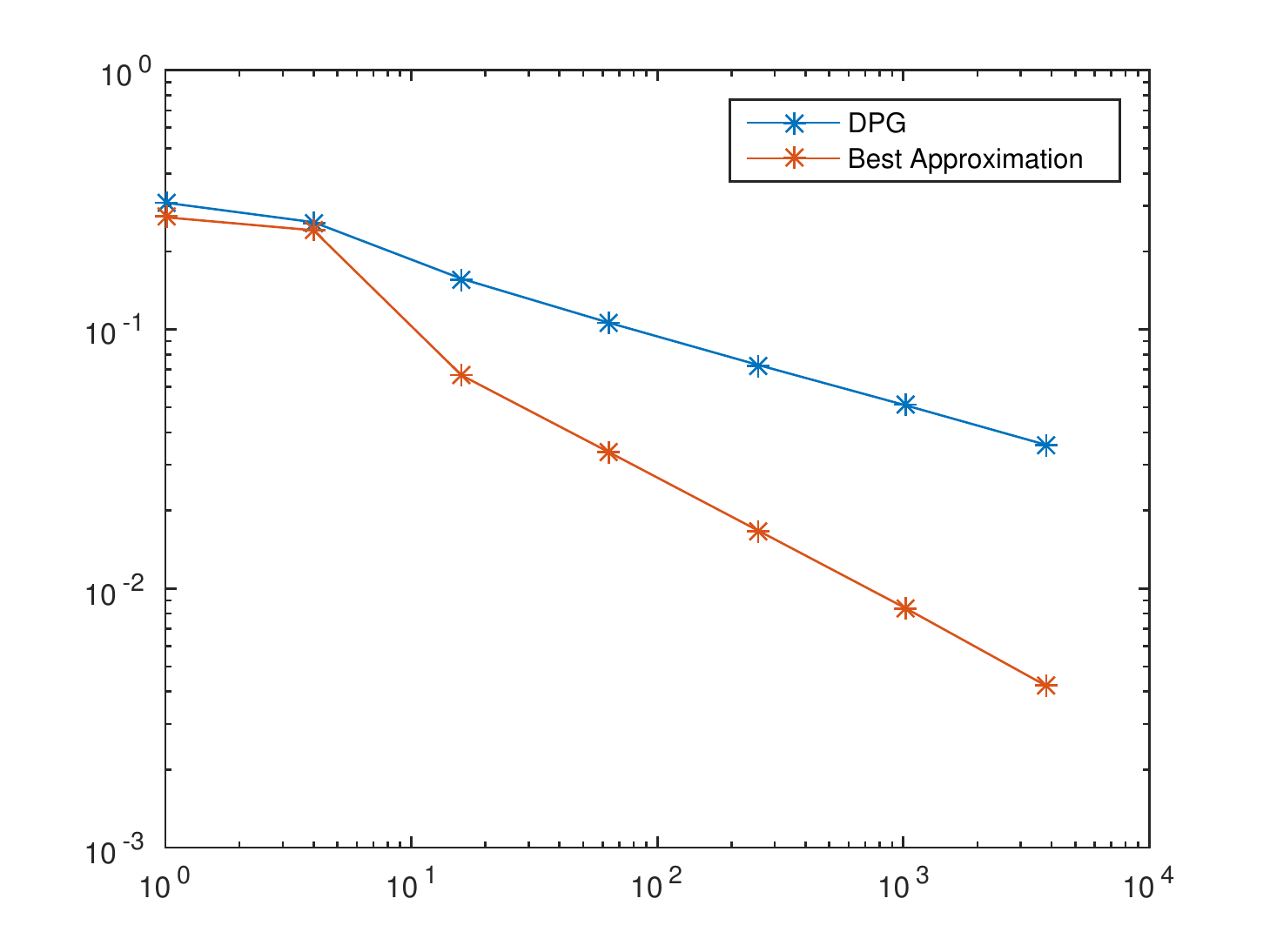}
\includegraphics[width=6cm]{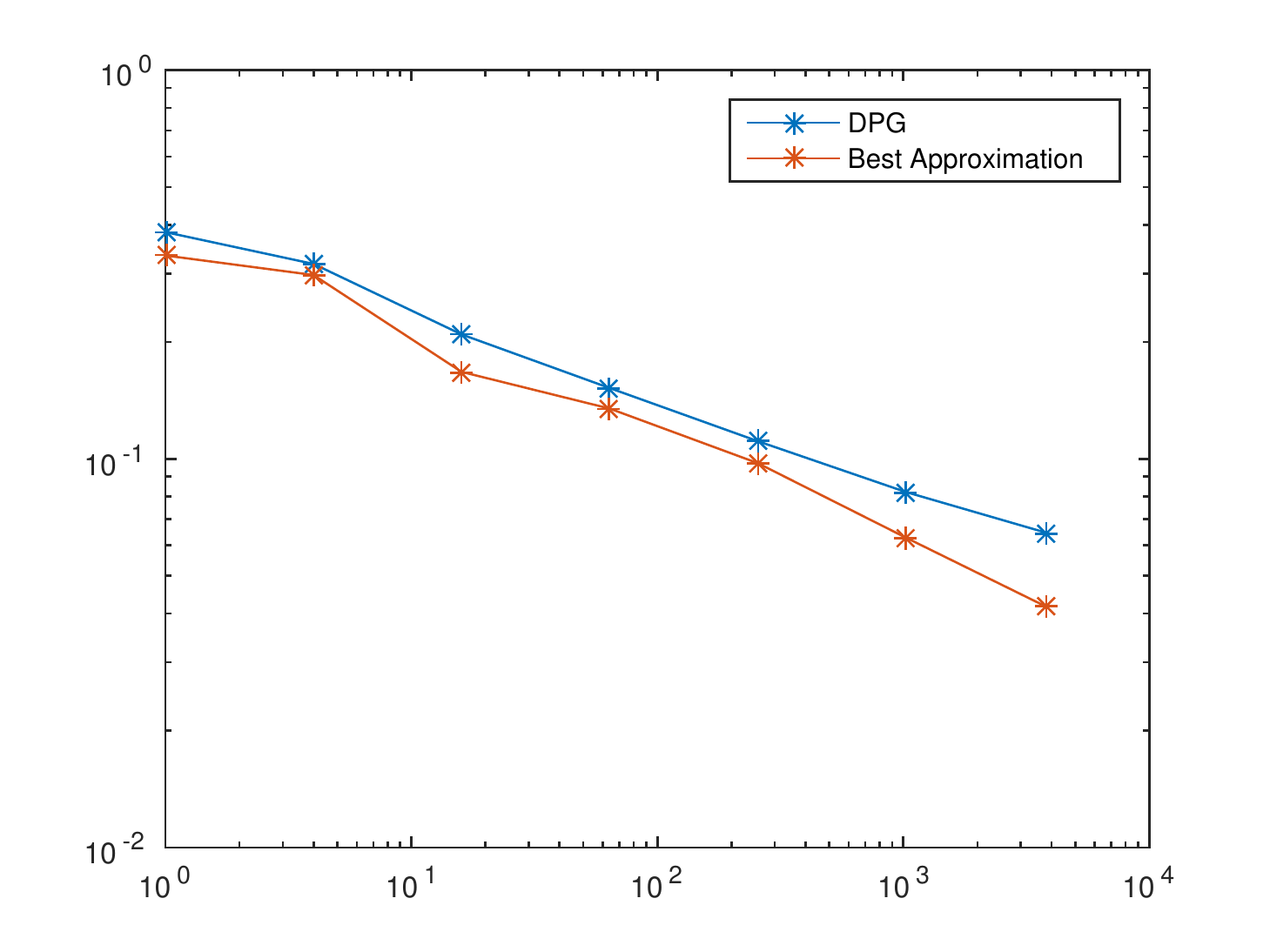}
\caption{$L_2(\Omega)$-error in $u_H$ and that in the best approximation versus $1/h^2$, for the discontinuous $f$ from \eqref{92}, ${\bf b}=(1,1)^\top$ (left) and ${\bf b}=(1,\frac{1}{16})^\top$ (right).}
\label{error2}
\end{figure}

To deal with this difficulty, we replaced the trial space for $\theta$ by the space of discontinuous polynomials $\prod_e \cP_1(e)$ with $e$ running over all edges of the mesh skeleton without the inflow edges, and determined the new test space ${\mathcal T}^h(\U^H)$ from \eqref{93} again with $\V^h=\prod_{K \in \Omega_h} \cP_{2}(K)$.
 With this modification, the curve of the $L_2(\Omega)$-error in the resulting $u_H$ is indistinguishable from that of the error in the best approximation.
Since $\prod_e \cP_1(e) \not\subset H_{0,\Gamma_{\!-}}({\bf b};\Omega)$ we are now dealing with a {\em nonconforming} DPG method. 

\begin{remark}
 This discontinuous trial space was already considered in the first paper \cite{64.14} in which such DPG discretizations (there called DPG-A) for the transport problem were considered.
 Since instead of $H_{0,\Gamma_{\!-}}({\bf b};\partial\Omega_h)$, there the space $L_2(\partial\Omega_h)$ was considered as the space for the traces $\theta$, the use of a nonconforming trial space was unintended.
 In \cite{138.22} one can find an analysis of a nonconforming DPG discretization for the Poisson problem.
 The analysis of the above nonconforming DPG discretization of the transport problem is  open.
 \end{remark}

In our third experiment we took ${\bf b}({\bf x})=(x_2,-x_1)^\top$, $f = 0$, and the inhomogeneous boundary condition $u=g$ on $\Gamma_{\!-}$, where $g(x_1,1)= 0$, and 
$g(0,x_2) =\left\{
\begin{array}{l}
 1, \  x_2 \ \geq \ \frac{1}{4},\\
 0, \  x_2 \ < \ \frac{1}{4} .
 \end{array} \right.
$
To implement this inhomogeneous boundary condition, following the second approach discussed in Remark~\ref{rem1} we solved $(u^H,\theta^H) \in \U^H$ from
$$
b_h(u^H,\theta^H;v^h) = -\int_{\partial\Omega_h} \llbracket v^h {\bf b} \rrbracket \bar{g} \,d{\bf s} \qquad(v^h \in  {\mathcal T}^h(\U^H)),
$$
with $\bar{g} \in H({\bf b};\Omega)$ being an extension of $g$. We took $\bar{g}({\bf x})=\left\{\begin{array}{l}
1, \  |{\bf x}| \ \in  \ [\frac{1}{4},1],\\ 
0, \  \text{elsewhere,}  \end{array} \right.
$
which in this case equals the exact solution.

In this experiment, we employed an adaptive refinement strategy, that we implemented using the package iFEM by L. Chen (\cite{ifem}).
By an application of Theorem~\ref{th4} and Riesz' representation theorem, $r \in H({\bf b};\Omega_h)$, defined by
$$
\langle r,v\rangle_{H({\bf b};\Omega_h)}=\int_{\partial\Omega_h} \llbracket v {\bf b} \rrbracket \bar{g} \,d{\bf s}-b_h(u^H,\theta^H;v) \quad(v \in H({\bf b};\Omega_h)),
$$
satisfies $\|r\|^2_{H({\bf b};\Omega_h)} \eqsim \|u-u^H\|_{L_2(\Omega)}^2+\|\theta-\theta^H\|^2_{H_{0,\Gamma_{\!-}}({\bf b};\partial\Omega_h)}$.
We approximated $r$ by the solution $\tilde{r} \in \V^h$  of 
$$
\langle \tilde{r},v^h\rangle_{H({\bf b};\Omega_h)}=\int_{\partial\Omega_h} \llbracket v^h {\bf b} \rrbracket \bar{g} \,d{\bf s}-b_h(u^H,\theta^H;v^h) \quad(v^h \in \V^h).
$$
Based on the decomposition $\|\tilde{r}\|^2_{H({\bf b};\Omega_h)}=\sum_{K \in \Omega^H} \|\tilde{r}\|^2_{H({\bf b};K)}$, as local error indicators we used $\{\|\tilde{r}\|^2_{H({\bf b};K)}\colon K \in \Omega^H\}$  to drive the common adaptive finite element method (AFEM) with D\"{or}fler marking parameter $\vartheta=\frac{1}{2}$.
Examples of a resulting mesh and approximate solution are given in Figure~\ref{adapt1}, and the $L_2(\Omega)$-errors vs. the number of unknowns are illustrated in Figure~\ref{error3}.
\begin{figure}[h]
\includegraphics[width=6cm]{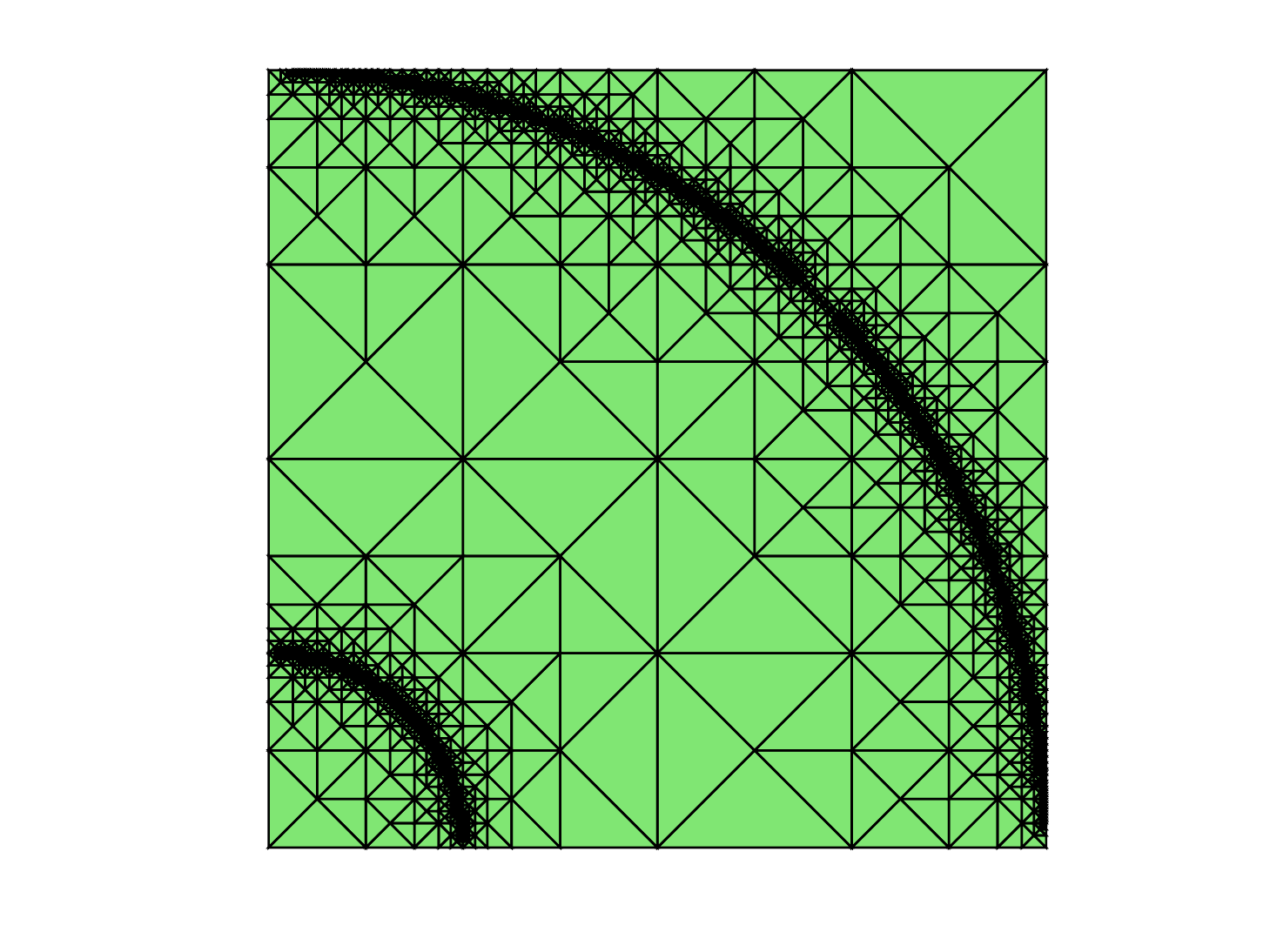}
 \includegraphics[width=6cm]{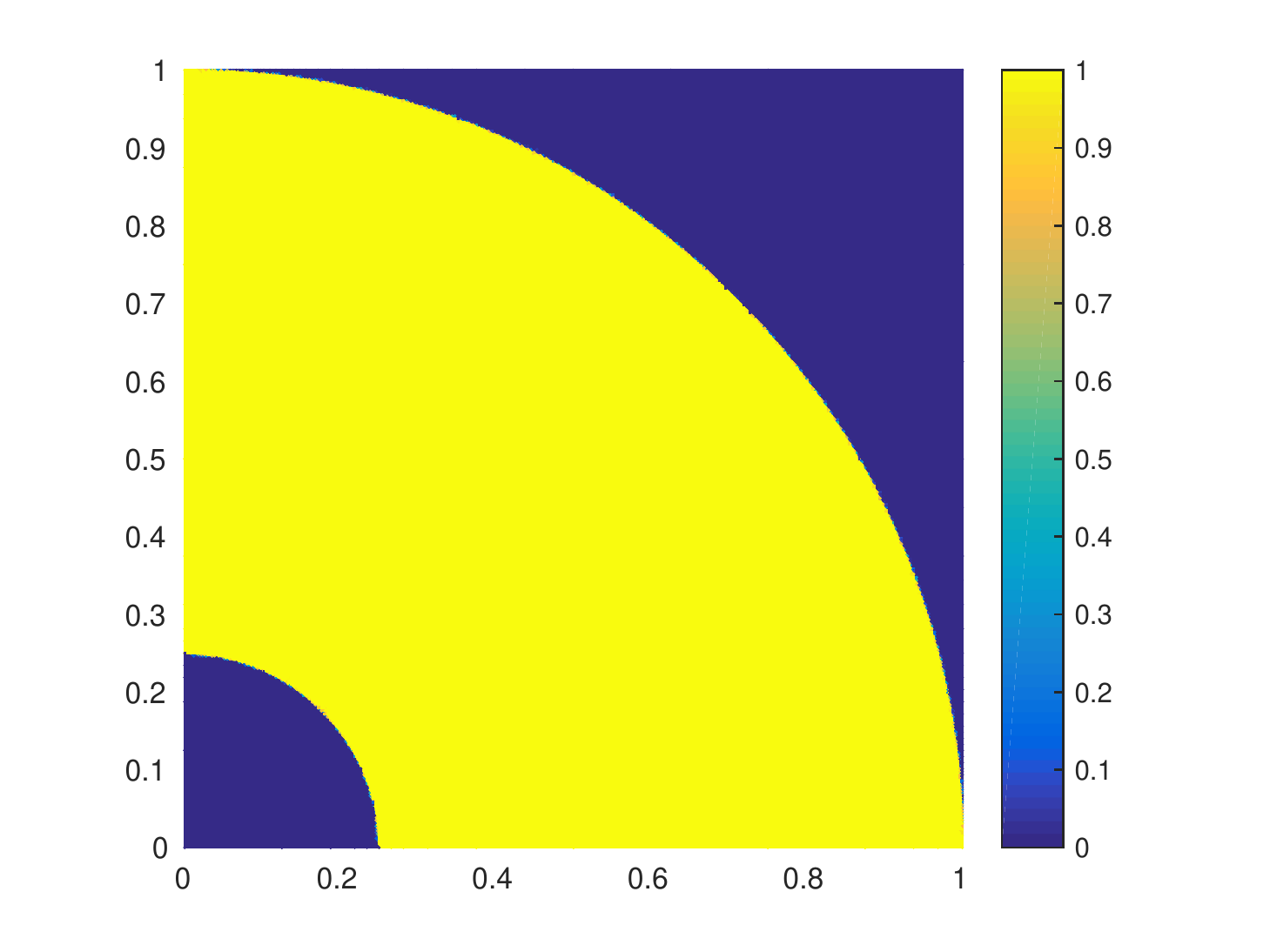}
\caption{Mesh generated after some iterations (left) and the approximate solution (right) for the third experiment.}
\label{adapt1}
\end{figure}
\begin{figure}[h]
\includegraphics[width=6cm]{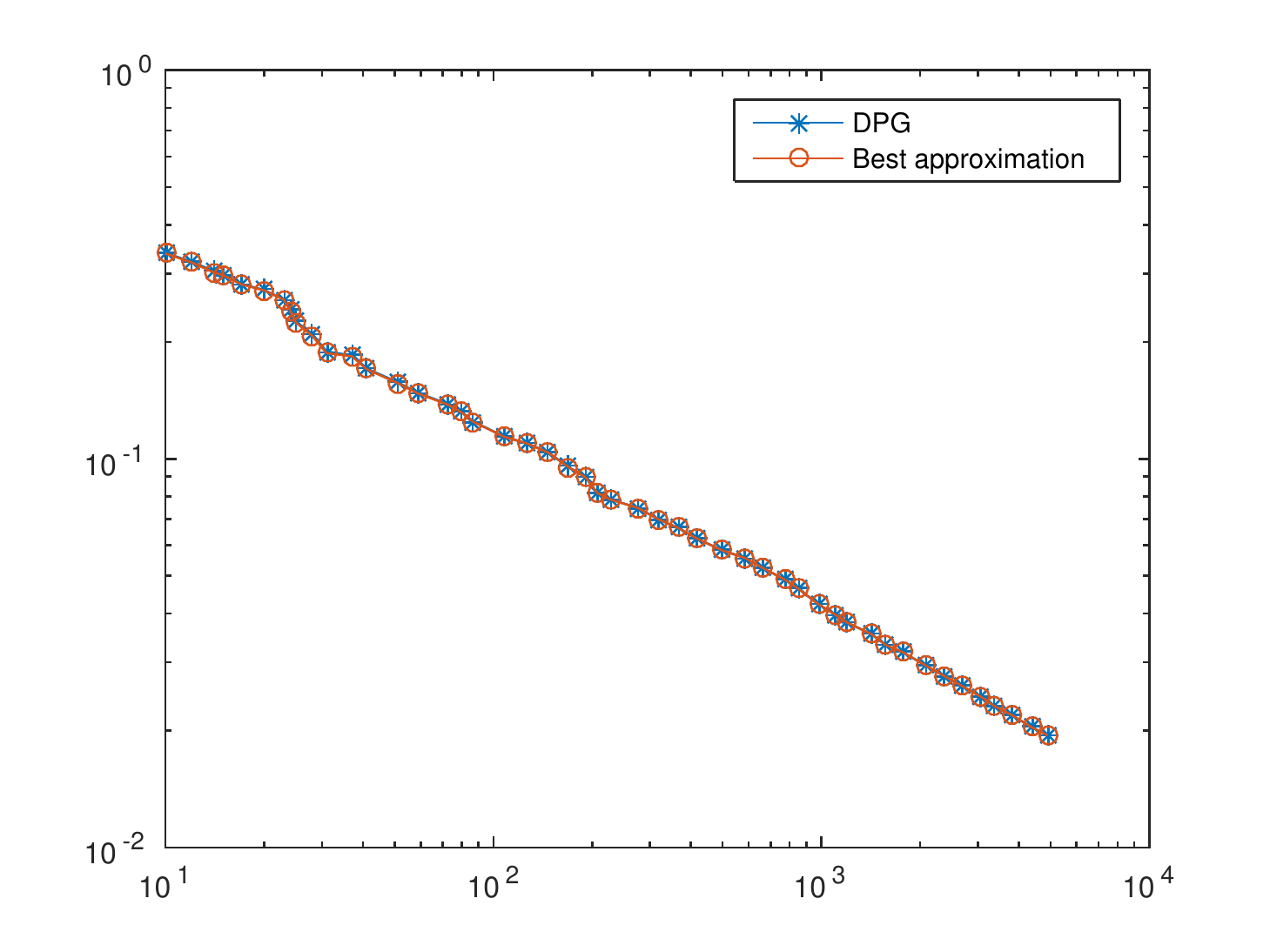}
\caption{$L_2(\Omega)$-error in $u_H$ and that in the best approximation versus the number of triangles in the mesh for the third experiment.}
\label{error3}
\end{figure}

\section{Conclusion}

For a family of uniformly shape regular partitions $\Omega_H, H\in \mathcal{I}$ and given piecewise polynomial trial spaces
on $\Omega_H$ and on the skeleton $\partial\Omega_H$, we have constructed  {\em uniformly inf-sup stable} (with respect to $H\in \mathcal{I}$)
DPG discretizations for linear transport equations with variable convection fields by
associating with each cell $K'\in \Omega_H$ a {\em piecewise polynomial }test space $\V_{K'}$ on a subgrid $\Omega_h|_{K'}$
with the following properties. 
The polynomial degree of each  $\V_{K'}$ exceeds the degree of the trial functions by one and
the refinement depth of each subgrid $\Omega_h|_{K'}$ is uniformly bounded. The stability implies that the DPG scheme provides near-best approximations from the trial space as well as uniform error-residual relations that form essential prerequisits for a posteriori error control and adaptive refinement strategies, see \eqref{GenCea}, \eqref{error-res}. The control of the polynomial degrees in the test space as well as the subgrid  refinement depth entail 
an asymptotically optimal complexity scaling since the size of the linear systems stays proportional to the dimensions of the trial spaces.
To our knowledge this is the first instance of a DPG stability result with the desired scaling properties except for \cite{75.61} for the elliptic case.
However, while the actual dimension of the local test spaces in \cite{75.61} could be made  concrete, the specification 
of the actual subgrid
refinement depth required in Theorem \ref{th1} would require knowledge of or good estimates for the various unspecified
constants entering the analysis. The strategy for proving Theorem \ref{th1} is necessrily entirely different from the elliptic case
and the analysis   indicates that  realizing a uniform inf-sup stability 
while keeping the dimensions of the local test spaces uniformly bounded, is not for granted when dealing with non-elliptic problems.
Several consequences of the findings in the present paper such as rigorous computable a posteriori error bounds or applications 
in more complex problem settings such as kinetic models will be addressed in forthcoming work.

\def\cprime{$'$}


\begin{thebibliography}{DSMMO04}

\bibitem[BM84]{BaMo84}
\newblock J.W. Barrett, and K.W. Morton,
Approximate symmetrization and Petrov-Galerkin methods for diffusion-convection problems.
\newblock {\em Comput. Method. Appl. M.,} {45} (1984), 97--12


\bibitem[BS14a]{35.856}
D.~Broersen and R.P. Stevenson.
\newblock A {Petrov-Galerkin} discretization with optimal test space of a
  mild-weak formulation of convection-diffusion equations in mixed form.
\newblock {\em IMA. J. Numer. Anal.}, 2014.
\newblock doi: 10.1093/imanum/dru003.

\bibitem[BS14b]{35.8565}
D.~Broersen and R.P. Stevenson.
\newblock A robust {Petrov-Galerkin} discretization of convection-diffusion
  equations.
\newblock {\em Comput. Math. Appl.}, 2014.
\newblock doi:10.1016/j.camwa.2014.06.019.

\bibitem[C09]{ifem}
L.~Chen. 
\newblock iFEM: An integrated finite element method package in MATLAB. 
\newblock Technical Report, University of California at Irvine, 2009.

\bibitem[CDW12]{45.44}
A.~Cohen, W.~Dahmen, and G.~Welper.
\newblock Adaptivity and variational stabilization for convection-diffusion
  equations.
\newblock {\em {ESAIM: Mathematical Modelling and Numerical Analysis}},
  46:1247--1273, 2012.

\bibitem[DG11]{64.14}
L.~Demkowicz and J.~Gopalakrishnan.
\newblock A class of discontinuous {P}etrov-{G}alerkin methods. {II}. {O}ptimal
  test functions.
\newblock {\em Numer. Methods Partial Differential Equations}, 27(1):70--105,
  2011.

\bibitem[DHSW12]{58.3}
W.~Dahmen, C.~Huang, Ch. Schwab, and G.~Welper.
\newblock Adaptive {P}etrov-{G}alerkin methods for first order transport
  equations.
\newblock {\em SIAM J. Numer. Anal.}, 50(5):2420--2445, 2012.

\bibitem[DPW]{DPW}
W. Dahmen, C. Plesken, G. Welper, Double Greedy Algorithms: Reduced Basis Methods for Transport Dominated Problems,
ESAIM: Mathematical Modelling and Numerical Analysis, 48(3) (2014), 623--663.
DOI 10.1051/m2an/2013103,\\ { http://arxiv.org/abs/1302.5072}.

\bibitem[DSMMO04]{64.573}
H.~De~Sterck, T.A. Manteuffel, S.F. McCormick, and L.~Olson.
\newblock Least-squares finite element methods and algebraic multigrid solvers
  for linear hyperbolic {PDE}s.
\newblock {\em SIAM J. Sci. Comput.}, 26(1):31--54, 2004.

\bibitem[GQ14]{75.61}
J.~Gopalakrishnan and W.~Qiu.
\newblock An analysis of the practical {DPG} method.
\newblock {\em Math. Comp.}, 83(286):537--552, 2014.

\bibitem[HKS14]{138.22}
N.~Heuer, M.~Karkulik, and F.-J. Sayas.
\newblock Note on discontinuous trace approximation in the practical {DPG}
  method.
\newblock {\em Comput. Math. Appl.}, 68(11):1562--1568, 2014.
\end{thebibliography}
\end{document}